\documentclass[a4paper,10pt]{article}
\usepackage[utf8]{inputenc}

\usepackage{geometry}
 \newtheorem{remark}{Remark}

\newtheorem{proof}{Proof}

\newtheorem{algorithm}{Algorithm}
\newtheorem{proposition}{Proposition}
\newtheorem{corollary}{Corollary}

\usepackage{hyperref}
\usepackage{bm}
\usepackage{float}
\usepackage{graphicx}
\usepackage{xcolor}
\usepackage{comment}
\usepackage{framed}
\usepackage{lipsum}
\usepackage{amsfonts}
\usepackage{epstopdf}
\usepackage{algorithmicx}
\usepackage{algpseudocode}

\usepackage{amsopn}
\usepackage[english]{babel}
\usepackage{moreverb}

\usepackage{lineno}
\usepackage{amsmath, bm}

\usepackage{caption}
\usepackage{booktabs}
 
\usepackage{color}

\usepackage{amssymb}
\usepackage{latexsym}
\usepackage{fancyhdr}
\usepackage{newlfont}
\usepackage{textcomp}
\usepackage{float}
\usepackage{multirow}

\usepackage{enumitem}

\newcommand{\mcF}{\mathcal{F}}
\newcommand{\mcE}{\mathcal{E}}

\newcommand{\mcX}{\mathcal{X}}

\newcommand{\mcL}{\mathcal{L}}

\newcommand{\mcA}{\mathcal{A}}
\newcommand{\mcAe}{{\mcA_\epsilon}}
\newcommand{\mcT}{\mathcal{T}}

\newcommand{\mbR}{\mathbb{R}}

\newcommand{\mbRn}{{\mathbb{R}^n}}

\newcommand{\Omg}{{\Omega}}
\newcommand{\Gam}{{\Gamma}}

\newcommand{\gam}{{\gamma}}
\newcommand{\game}{{\gamma_\epsilon}}
\newcommand{\mude}{{\mu_{\delta,\epsilon}}}

\def \xb{\bm{x}}
\def \yb{\bm{y}}

\def \ue{{u_\epsilon}}

\def \intl{\int\limits}
\def\Omg{{\Omega}}

\newcommand{\eugenio}[1]{\textcolor{red}{\textbf{#1}}}

\title{\bf Efficient quadrature rules for finite element discretizations of nonlocal equations}
\author{%
  Eugenio Aulisa \footnote{Department of Mathematics and Statistics, Texas Tech University, TX, USA.}%
   \and
  Giacomo Capodaglio\footnote{Computational Physics and Methods Group, Los Alamos National Laboratory, NM, USA}%
  \and 
   Andrea Chierici \footnote{Department of Industrial Engineering, University of Bologna, Bologna, Italy}%
 \and
   Marta D'Elia \footnote{Computational Science and Analysis, Sandia National Laboratories, CA, USA}%
  }
\date{}

\begin{document}
\maketitle

\begin{abstract}
In this paper we design efficient quadrature rules for finite element discretizations of nonlocal diffusion problems with compactly supported kernel functions. Two of the main challenges in nonlocal modeling and simulations are the prohibitive computational cost and the nontrivial implementation of discretization schemes, especially in three-dimensional settings. In this work we circumvent both challenges by introducing a parametrized mollifying function that improves the regularity of the integrand, utilizing an adaptive integration technique, and exploiting parallelization. We first show that the ``mollified'' solution converges to the exact one as the mollifying parameter vanishes, then we illustrate the consistency and accuracy of the proposed method on several two- and three-dimensional test cases. Furthermore, we demonstrate the good scaling properties of the parallel implementation of the adaptive algorithm and we compare the proposed method with recently developed techniques for efficient finite element assembly.
\end{abstract} 
 
\section{Introduction}\label{sec:introduction}
Nonlocal equations have become the model of choice in applications where the global behavior of the system is affected by long-range forces at small scales. In particular, these equations are preferable to partial differential equations (PDEs) in presence of anomalous behavior, such as superdiffusion and subdiffusion, multiscale behavior, and discontinuities or irregularities in the solution that cannot be captured by classical models. For these reasons, nonlocal models are currently employed in several scientific and engineering applications including surface or subsurface transport 
\cite{Benson2000,Benson2001,Deng2004,Schumer2003,Schumer2001},
fracture mechanics
\cite{Ha2011,Littlewood2010,Silling2000},
turbulence
\cite{DiLeoni-2020,Pang2020},
image processing
\cite{Buades2010,DElia2019imaging,Gilboa2007,Lou2010}
and stochastic processes
\cite{Burch2014,DElia2017,Meerschaert2012,Metzler2000,Metzler2004}.

The most general form \cite{DElia2020Unified} of a nonlocal operator for a scalar function $u:\mbRn\to\mbR$ is given by
\begin{equation*}
\mcL u(\xb) = 2\int_\mbRn 
(u(\yb)-u(\xb)) \gamma(\xb,\yb)\,d\yb,
\end{equation*}
where $\gamma$, the kernel, is a compactly supported function over $B_\delta(\xb)$, the ball of radius $\delta$ centered at $\xb$. We refer to $\delta$ as horizon or interaction radius; this quantity determines the extent of the nonlocal interactions and represents the length scale of the operator. The integral form allows one to catch long-range forces within the length scale and reduces the regularity requirements on the solutions. It also highlights the main difference between nonlocal models and PDEs, i.e., the fact that interactions can occur at distance, without contact. The kernel $\gamma$ depends on the application and determines the regularity properties of the solutions; the choice of its parameters or functional form is among the most investigated open questions in the current nonlocal literature \cite{Pang2020,burkovska2020,DElia2014DistControl,DElia2016ParamControl,Gulian2019,Pang2019fPINNs,Pang2017discovery,Xu2020learning,You2020Regression,You2020aaai}. In this work we limit ourselves to smooth integrable kernels since the treatment of more complex functions is not germane to the issues investigated in this paper, as clarified later on. 

The integral nature of the operator poses several modeling and numerical challenges including the treatment of nonlocal interfaces \cite{Alali2015,Capodaglio2019}, the prescription of nonlocal boundary conditions \cite{Cortazar2008,DEliaNeumann2019} and the design of efficient discretization schemes and numerical solvers \cite{AinsworthGlusa2018,Capodaglio2020DD,acta20,DEliaFEM2020,Pasetto2019,silling2005meshfree,Wang2010}. In fact, the numerical solution of nonlocal equations becomes prohibitively expensive when the ratio between the interaction radius and the discretization size increases. Even though the nonlocal literature offers several examples of meshfree, particle-type discretizations \cite{Chen2006meshless,parks2012peridigm,parks2010lammps,silling2005meshfree}, in this paper we focus on finite element (FE) methods. This allows us to easily deal with nontrivial domains, achieve high-order accuracy, and use mesh adaptivity. Furthermore, the nonlocal vector calculus \cite{du13} provides a means for a rigorous stability and convergence analysis of variational methods as it allows us to analyze nonlocal diffusion problems in a similar way as elliptic PDEs \cite{du12}.

When cast in a variational form, the nonlocal problem associated with the operator $\mcL$ results in a bilinear form characterized by the following double integral
\begin{equation*}
\int_{\Omega\cup\Gamma}  \int_{(\Omega\cup\Gamma)\cap B_\delta(\xb)}\big(u(\yb) - u(\xb)\big)\big(\varphi(\yb) - \varphi(\xb)\big)\gamma(\xb,\yb)d\yb  d\xb 
\end{equation*}
where we explicitly reported the domain of integration in the inner integral. Here, $\Omega\in\mbRn$ and $\Gamma$ are the domain of interest and the corresponding ``nonlocal boundary'' and $\varphi$ is an appropriate test function. Thus, the variational setting introduces further computational challenges. Not only do we have to numerically evaluate a double integral, but the integrand function is discontinuous, due to the compact support of $\gamma$ and to the fact that, in FE settings, the tests functions are also compactly supported. 

The paper by D'Elia et al.\cite{DEliaFEM2020} thoroughly describes the challenges associated with nonlocal FE discretizations and proposes approximation techniques for efficient and accurate implementations. In particular, the authors introduce ``approximate balls'' that facilitate the assembly procedure by substituting the Euclidean ball $B_\delta(\xb)$ with suitable polygonal approximations. They also suggest a set of quadrature rules for the outer and inner integration and analyze the convergence properties of the resulting scheme. With the same spirit, in this work, we propose an alternative way to efficiently evaluate the integral above by circumventing the issue of integrating a truncated function. The key idea of this paper is the introduction of a mollifier \cite{mousavi2012efficient} to approximate the discontinuous kernel function; by doing so, the new, approximate, and parameterized kernel is a smooth function for which standard Gaussian quadrature rules can be employed over every element without compromising their accuracy. Additionally, we introduce adaptive quadrature rules for the numerical integration of the outer integral. In fact, contrary to intuition, sophisticated integration techniques for the outer integral are required in order to prevent the quadrature error from exceeding the FE one \cite{DEliaFEM2020}.

The main contributions of this paper are
\begin{itemize}
    \item The introduction of a parametrized, smooth, approximate kernel, by means of a mollifier, that yields a smooth integrand over every element. This allows us to avoid the tedious and impractical task of determining the intersections of the ball with the elements, and hence represents a major advantage of our method in three-dimensional simulations. 
    \item The design of adaptive quadrature rules for the outer integral and of a parallel algorithm for efficient simulations.
    \item The theoretical proof and numerical illustration of the convergence of the approximate, mollified solution to the analytic one as the mollified kernel approaches $\gamma$.
    \item The numerical illustration of the convergence of the mollified solution to the exact one as we refine the mesh and a numerical study of the dependence of the convergence behavior with respect to the parameters. 
    \item The demonstration via two-dimensional and three-dimensional numerical tests of the scalability of our algorithm.
\end{itemize}

\paragraph{Outline of the paper} 
In the following section we define the notation that is used throughout the paper and recall important results on nonlocal calculus. In Section \ref{sec:weak-form} we introduce the mollifier function and the associated approximate, parametrized weak form of the nonlocal diffusion problem. We also analyze the convergence of the solution of the latter to the original weak solution. In Section \ref{sec:fem} we describe the nonlocal FE discretization, with special focus on the assembly procedure, and briefly recall its challenges. In Section \ref{sec:adaptivity} we introduce adaptive quadrature rules for the numerical integration of the outer integral. In Section \ref{sec:numerics} we illustrate our theoretical results via two- and three-dimensional tests. We also discuss a parallel implementation of the FE assembly procedure and show the corresponding scaling results. Finally, we summarize our contributions in Section \ref{sec:conclusion}.

%%%%%%%%%%%%%%%%%%%%%%%%%%%%%%%%%%%%%%%%%%%%%%%%%%%%%%%%%%%%%
%%%%%%%%%%%%%%%%%%%%%%%%%%%%%%%%%%%%%%%%%%%%%%%%%%%%%%%%%%%%%
\section{Preliminaries}\label{sec:preliminaries}
Let $\Omega \in \mathbb{R}^n$ be open and bounded, $n=1,2,3$.
Given some $\delta>0$, we define the interaction domain $\Gamma$ of $\Omega$ as the set of all points not in $\Omega$ that are within a $\delta$ distance from points in $\Omega$, i.e.
\begin{equation}\label{eq:interaction-domain}
\Gamma = \{\yb\in \mathbb R^n \setminus \Omega:  |\xb-\yb|\leq\delta \; \text{for some} \; \xb\in\Omega\},
\end{equation}
see Figure \ref{fig1} for a graphical example in $\mathbb{R}^2$. Note that $\Gamma$ depends on $\delta$ even if it is not explicitly indicated. Let $\gamma: \mathbb{R}^n \times \mathbb{R}^n \rightarrow \mathbb{R}$ be an integrable nonnegative symmetric kernel that is also radial\footnote{For a discussion on nonpositive kernels and nonsymmetric kernels, see \cite{Mengesha-sign-changing} and \cite{DElia2017}, respectively.},  namely $\gamma(\xb,\yb) = \gamma(\yb,\xb)$ and $\gamma(\xb,\yb) = \gamma(|\xb-\yb|)$. We also assume that $\gamma$ has bounded support over the ball of radius $\delta$ centered at $\xb$, i.e. $B_\delta(\xb)$. For a scalar function $u:\mbRn\to\mbR$ we
define the nonlocal Laplacian as 
\begin{equation}\label{eq:L}
\mcL u(\xb) = 2\int_\mbRn 
(u(\yb)-u(\xb)) \gamma(\xb,\yb)\,d\yb.
\end{equation}
\begin{figure}[t!]
   \centering
   \includegraphics[width=2in]{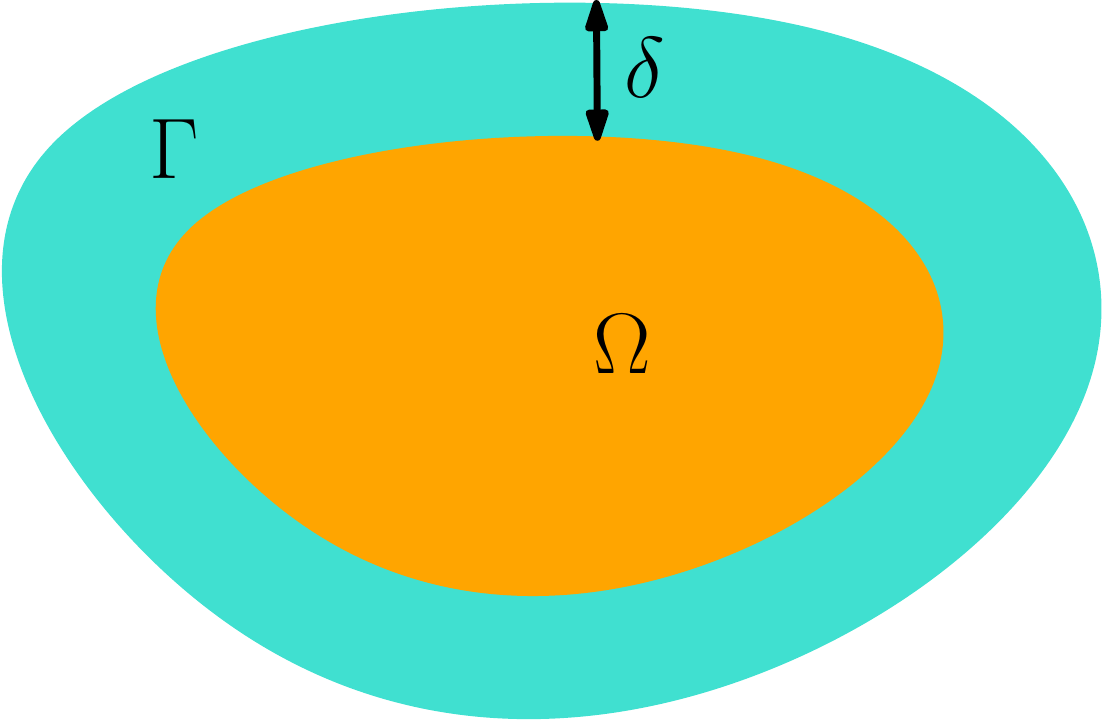}
\caption{Example of domain $\Omega$ with associated interaction domain $\Gamma$.}
   \label{fig1}
\end{figure}
The strong form of a nonlocal Poisson problem is then given by: for $f:\Omega\to\mbR$, and $g:\mbRn\setminus\Omega\to\mbR$, find $u$ such that
\begin{equation}\label{eq:poisson}
\left\{\begin{aligned}
-\mcL u (\xb) & =f(\xb),  
&\quad \xb\in\Omega\\
u(\xb) &= g(\xb),         
&\quad \xb\in \Gamma
\end{aligned}\right.
\end{equation}
where the second condition in \eqref{eq:poisson} is the nonlocal counterpart of a Dirichlet boundary condition for PDEs and it is referred to as {\it Dirichlet volume constraint}\footnote{For definition and analysis of Neumann volume constraints we refer to \cite{Du2012} and for its numerical treatment we refer to, e.g.,  \cite{DEliaNeumann2019}.}. Such condition is required \cite{du12} to guarantee the well-posedness of \eqref{eq:poisson}. 
%For simplicity and without loss of generality we analyze the homogeneous case $g=0$; all the results below can be extended to the non-homogeneous case using ``lifting'' arguments (see, e.g., \cite{DElia2014DistControl}).
The weak form of the Poisson problem is obtained by multiplying the first equation in \eqref{eq:poisson} by a test function $\varphi=0$ in $\Gamma$ and by applying the nonlocal first Green's identity \cite{du13}. This yields
\begin{equation}\label{eq:poisson-weak-extended}
\begin{aligned}
0 & =  \int\limits_\Omega (-\mcL u -f)\, \varphi \,d\xb \\
& = 
\iint\limits_{(\Omega\cup\Gamma)^2} (u(\yb)-u(\xb))(\varphi(\yb)-\varphi(\xb))\,\gamma(\xb,\yb) \,d\yb\,d\xb 
- \int\limits_\Omega f(\xb)\, \varphi(\xb) \,d\xb,
\end{aligned}
\end{equation}
Then, the weak form of the nonlocal diffusion problem reads as follows. For $f\in V'$ and $g\in V_\Gamma$, find $u\in V$ such that
\begin{equation}\label{eq:poisson-weak}
\mcA(u,v) = \mcF(v), \;\;\forall\, v\in V_0,
\quad \hbox{subject to} \; u=g \; \hbox{ in} \; \Gamma,
\end{equation}
where
\begin{equation}\label{eq:A-F}
\begin{aligned}
\mcA(u,\varphi) & = \iint\limits_{(\Omega\cup\Gamma)^2} (u(\yb)-u(\xb))(\varphi(\yb)-\varphi(\xb))\gamma(\xb,\yb) \,d\yb\,d\xb,\\
\mcF(\varphi) & =\int_\Omega f(\xb)\, \varphi(\xb) \,d\xb,
\end{aligned}
\end{equation}
and where the function spaces are defined as
\begin{equation}\label{eq:V}
\begin{aligned}
V & =\{\varphi\in L^2(\Omega\cup\Gamma): |||\varphi|||<\infty\;
\hbox{and} \; v|_{\mbRn\setminus\Omega} = 0\}\\
V_0 & =\{\varphi\in V: \varphi|_\Gamma = 0\},\\
V_\Gamma &=\{p:\Gamma\to\mbR: \,\exists\, \varphi\in V \,\hbox{such that} \; \varphi|_\Gamma=p\}.
\end{aligned}
\end{equation}
Here, the {\it energy} semi-norm $|||\cdot|||$ is defined as
\begin{equation}\label{eq:unweighted-energy}
|||\varphi|||^2 = \mcA(\varphi,\varphi),
\end{equation}
the space $V'$ is the dual space of $V$ and $V_\Gamma$ is a nonlocal trace space. Note that since the kernel is integrable and translation invariant, the energy semi-norm is a norm in the constrained space $V_0$ and satisfies a Poincar\'e inequality \cite{du12}. Furthermore, by construction, the bilinear form $\mcA(\cdot,\cdot)$ defines an inner product on $V_0$ and it is continuous and coercive with respect to the energy norm $|||\cdot|||$. Finally, the latter is equivalent to the $L^2$ norm; this allows us to establish an equivalence relationship between $V$ and $L^2(\Omega\cup\Gamma)$. Together with the continuity of $\mcF$, these facts yield the well-posedness of the weak form \eqref{eq:poisson-weak} \cite{du12}.

As commonly done in the PDE context, we recast the problem in $V_0$ by simply rewriting the solution as $u = w+\widetilde g$, where $w\in V_0$ and $\widetilde g\in V$ is an extension of $g$ to zero into $\Omega\cup\Gamma$, known in the FE framework as a lifting function. Thus, equation \eqref{eq:poisson-weak} can be rewritten in terms of $w$ as follows
\begin{align}
&\iint\limits_{(\Omega\cup\Gamma)^2}  \left(w(\xb)-w(\yb)\right)\left(\varphi(\xb)-\varphi(\yb)\right)\gamma(\xb,\yb)  d\yb\,d\xb = \int_{\Omega} f(\xb) \varphi(\xb) d\xb \nonumber \\
& \quad+ \iint\limits_{(\Omega\cup\Gamma)^2}  \left(\widetilde g(\xb)-\widetilde g(\yb)\right)\left(\varphi(\xb)-\varphi(\yb)\right)\gamma(\xb,\yb)  d\yb\,d\xb,\quad  \forall \varphi \in V_0,
\label{eq:weak-homo-extended}
\end{align}
or, equivalently,
\begin{align}
\mcA(w,\varphi)=\widetilde \mcF(\varphi)\quad  \forall \varphi \in V_0.
\label{eq:weak-homo}
\end{align}
The latter is useful for implementation purposes as it allows us to automatically take into account the presence of a non-homogeneous Dirichlet volume constraint. 

%%%%%%%%%%%%%%%%%%%%%%%%%%%%%%%%%%%%%%%%%%%%%%%%%%%%%%%%%%%%%
%%%%%%%%%%%%%%%%%%%%%%%%%%%%%%%%%%%%%%%%%%%%%%%%%%%%%%%%%%%%%
\section{Weak form approximation}\label{sec:weak-form}
We introduce a parametrized approximation of the bilinear form $\mcA$ defined in \eqref{eq:A-F} with the purpose of obtaining a weak problem that is computationally less challenging. Specifically, the approximated bilinear form is associated with a parametrized kernel function that is still integrable, radial, and compactly supported, but not discontinuous in $\Omega\cup\Gamma$. This fact makes the numerical integration of the inner integral in, e.g., \eqref{eq:poisson-weak}, a much simpler task, compared to the case of discontinuous kernel functions. 

\smallskip
For simplicity of exposition, we rewrite the ``exact'' kernel $\gam$ as  
\begin{equation}\label{eq:gamma}
\gam(\xb,\yb)=C_\delta \, \eta(\xb,\yb)
\,\mcX(\yb\in B_\delta(\xb))
\end{equation}
where $C_\delta$ is a scaling constant that guarantees that the nonlocal operator $\mcL$ associated with $\gam$ is such that $\mcL\to\Delta$ as $\delta\to 0$. Clearly, by definition, $\eta(\xb,\yb)=\eta(| \xb-\yb|)$.
Given $\varepsilon\in\mbR^+$, we approximate the bilinear form $\mathcal{A}(\cdot,\cdot)$ defined in \eqref{eq:A-F} with the parametrized bilinear form $\mathcal{A}_{\epsilon}(\cdot,\cdot)$ obtained by replacing the kernel $\gam(\xb,\yb)$ with
\begin{equation}
\label{eq:approx-kernel}
\game(\xb,\yb)=C_{\delta,\epsilon} \eta(\xb,\yb)
\mude(\xb,\yb),
\end{equation} 
where $\mude$ is an appropriately scaled mollifier function. Inspired by the mollifier function introduced in \cite{mousavi2012efficient}, for $\varepsilon < \delta$ we define $\mude: \mathbb{R}^n \times \mathbb{R}^n \rightarrow \mathbb{R}$ as the following radial function  
\begin{align}
\label{eq:mollifier}
&\mude(|\xb-\yb|)= 
\left\{
\begin{array}{l l}
1 & \mbox{ for } 0\le |\xb-\yb| < \delta -\varepsilon\\[2mm]
\xi\left(\frac{ (\delta-\varepsilon) - |\xb-\yb|}{\varepsilon}\right) &\mbox{ for }  \delta -\varepsilon \le |\xb-\yb| \le \delta +\varepsilon  \\[2mm]
0 &\mbox{ for } |\xb-\yb| > \delta + \varepsilon
\end{array}
\right.\\[3mm]
&\xi(r)=\left(\frac{128}{256} + \frac{315}{256} r - \frac{420}{256} r^3 
+ \frac{378}{256}r^5 - \frac{180}{256} r^7+ \frac{35}{256}r^9 \right),\nonumber
\end{align}
where, for given $\epsilon>0$, $C_{\delta,\epsilon}$ is such that the nonlocal operator $\mcL_\epsilon$, associated with $\game$, converges to $\Delta$ as $\delta\to 0$.
The constant $C_{\delta,\epsilon}$ is also such that it converges to $C_\delta$ as $\epsilon \to 0$.
Furthermore, it follows from the definition of $\mude$ that
$$
\lim\limits_{\epsilon\to 0}\mude(\xb,\yb) = \mcX(\yb\in B_\delta(\xb)),
$$
which, together with the property of $C_{\delta,\epsilon}$, implies that $\game$ converges pointwise to $\gam$ as $\epsilon\to 0$. Note that the support of the mollifier is bigger than the one of the original kernel function $\gamma$ as it corresponds to $B_{\delta+\epsilon}(\xb)$. 
The parametrized bilinear form $\mathcal{A}_{\epsilon}(\cdot,\cdot)$ is therefore defined as follows
\begin{align}\label{eq:Aeps}
\mathcal{A}_{\epsilon}(u,\varphi) = 
\iint_{(\Omega\cup\Gamma)^2} 
\big(u(\yb) - u(\xb)\big)\big(\varphi(\yb) - \varphi(\xb)\big)\gamma_{\epsilon}(\xb,\yb)d\yb  d\xb ,
\end{align}
and consequently the approximate weak formulation of the nonlocal volume-constrained problem now reads: find $u_{\epsilon} \in V$ such that
\begin{equation}\label{eq:weak_nonloc_approx}
\begin{aligned}
    \mathcal{A}_{\epsilon}(u_{\epsilon},\varphi) = \mcF(\varphi) \quad \forall\, \varphi \in V_0
    \quad\mbox{subject to } u_{\epsilon}=g \mbox{ on } \Gamma. 
\end{aligned}
\end{equation}
Note that the weak formulation above is defined over the same function spaces of \eqref{eq:poisson-weak}. This is allowed because the parametrized kernel $\game$ belongs to the same class of kernels as $\gamma$. As a consequence, problem \eqref{eq:weak_nonloc_approx} is also well-posed.

%%%%%%%%%%%%%%%%%%%%%%%%%%%%%%%%%%%%%%%%%%%%%%%%%%%%%%%%%%%%%
\subsection{Convergence of the approximate weak solution}

We find a bound for the energy norm of the difference between solutions of the weak form \eqref{eq:poisson-weak} and \eqref{eq:weak_nonloc_approx}.
First, we note that by subtracting \eqref{eq:poisson-weak} from \eqref{eq:weak_nonloc_approx} we obtain
\begin{equation}\label{eq:equivalence-forms}
\mcA(u,\varphi) = \mcA_\epsilon(u_\epsilon,\varphi) 
\quad \forall\,\varphi\in V_0.
\end{equation}
Our ultimate goal is to find a bound for $|||u-u_\epsilon|||$, or, equivalently, for $|\mcA(u-u_\epsilon,u-u_\epsilon)|$, being $u$ and $u_\epsilon$ solutions of \eqref{eq:poisson-weak} and \eqref{eq:weak_nonloc_approx} respectively.
We first consider a generic test function; equality \eqref{eq:equivalence-forms} implies
\begin{displaymath}
\begin{aligned}
|\mcA(u-\ue,\varphi)|
& = |\mcA(u,\varphi)-\mcA(\ue,\varphi)| 
  = |\mcAe(\ue,\varphi)-\mcA(\ue,\varphi)| \\
& = \left|\;\;\iint\limits_{(\Omega\cup\Gamma)^2} (\ue(\xb)-\ue(\yb))(\varphi(\xb)-\varphi(\yb))(\gamma_\epsilon(\xb,\yb)-\gamma(\xb,\yb))\,d\yb\,d\xb \right|.
\end{aligned}
\end{displaymath}
By expanding the product, we have
\begin{displaymath}
\begin{aligned}
& \left|\;\;\iint\limits_{(\Omega\cup\Gamma)^2} (\ue(\xb)-\ue(\yb))(\varphi(\xb)-\varphi(\yb))
(\gamma_\epsilon(\xb,\yb)-\gamma(\xb,\yb))\,d\yb\,d\xb \right| \\
\leq & \iint\limits_{(\Omega\cup\Gamma)^2} |\ue(\xb)\varphi(\xb)|\,|\gamma_\epsilon(\xb,\yb)-\gamma(\xb,\yb)|\,d\yb\,d\xb \\
+ & \iint\limits_{(\Omega\cup\Gamma)^2} |\ue(\yb)\varphi(\yb)|\,|\gamma_\epsilon(\xb,\yb)-\gamma(\xb,\yb)|\,d\yb\,d\xb \\
+ & \iint\limits_{(\Omega\cup\Gamma)^2} |\ue(\xb)\varphi(\yb)|\,|\gamma_\epsilon(\xb,\yb)-\gamma(\xb,\yb)|\,d\yb\,d\xb \\
+ & \iint\limits_{(\Omega\cup\Gamma)^2} |\ue(\yb)\varphi(\xb)|\,|\gamma_\epsilon(\xb,\yb)-\gamma(\xb,\yb)|\,d\yb\,d\xb.
\end{aligned}
\end{displaymath}
By switching the order of integration and renaming dummy variables in the second and fourth terms above, we obtain
\begin{equation}\label{eq:bilin-form-estimate}
\begin{aligned}
& 2 \int\limits_{\Omega\cup\Gamma} |\ue(\xb)\varphi(\xb)|\int\limits_{\Omega\cup\Gam}
|\gamma_\epsilon(\xb,\yb)-\gamma(\xb,\yb)|\,d\yb\,d\xb\\
+ &\;2  \int\limits_{\Omega\cup\Gamma} |\varphi(\xb)|
\int\limits_{\Omega\cup\Gam}
|\ue(\yb)|\,|\gamma_\epsilon(\xb,\yb)-\gamma(\xb,\yb)|\,d\yb\,d\xb \\
\leq & \;2 c_1(\epsilon) \int\limits_{\Omega\cup\Gamma} |\ue(\xb)\varphi(\xb)|\,d\xb\\
+&\; 2 \int\limits_{\Omega\cup\Gamma} |\varphi(\xb)|
\|\ue\|_{L^2(\Omega\cup\Gam)} \|\gam_\epsilon-\gam\|_{L^2(\Omega\cup\Gam)} \,d\xb\\
\leq & \;2 c_1(\epsilon) \|\ue\|_{L^2(\Omega\cup\Gam)} \|\varphi\|_{L^2(\Omega\cup\Gam)} 
+ 2 c_2(\epsilon) |\Omega\cup\Gamma|^\frac12 \|\ue\|_{L^2(\Omega\cup\Gam)} \|\varphi\|_{L^2(\Omega\cup\Gam)},
\end{aligned}
\end{equation}
where we used the Cauchy-Schwarz inequality for the outer and inner integral for the first and second term, respectively, and where  
\begin{displaymath}
\begin{aligned}
c_1(\epsilon)& =\max_{\xb\in \Omg\cup\Gam} \int_{B_{\delta+\epsilon}(\xb)\cap(\Omega\cup\Gam)} |\gam_\epsilon(\xb,\yb)-\gam(\xb,\yb)|\,d\yb \leq \int_{B_{\delta+\epsilon}(\mathbf{0})} |\gam_\epsilon({\mathbf 0},\yb)-\gam({\mathbf 0},\yb)|\,d\yb,\\[2mm]
c^2_2(\epsilon)& =\max_{\xb\in \Omg\cup\Gam} \int_{B_{\delta+\epsilon}(\xb)\cap(\Omega\cup\Gam)} (\gam_\epsilon(\xb,\yb)-\gam(\xb,\yb))^2 \,d\yb
\leq \int_{B_{\delta+\epsilon}({\mathbf 0})}  (\gam_\epsilon({\mathbf 0},\yb)-\gam({\mathbf 0},\yb))^2 \,d\yb.
\end{aligned}
\end{displaymath}
We recall that, by definition, $\gam_\epsilon\to\gam$ pointwise as $\epsilon\to 0$; thus, $c_i(\epsilon)\to 0 $ as $\epsilon\to 0$, for $i=1,2$. Furthermore, thanks to the properties of $\eta$ and $\mude$, the integrals above are well-defined.

To obtain the final estimate, we consider $\varphi=u-\ue$ and recall that for the kernels considered in this work the energy norm $|||\cdot|||$ is equivalent to the $L^2$ norm. In particular there exists a positive constant $C_{eq}$ such that $\|\varphi\|_{L^2(\Omega\cup\Gamma)}\leq |||\varphi|||$. Thus, we have the following estimate
\begin{displaymath}
\begin{aligned}
\|u-\ue\|^2_{L^2(\Omg\cup\Gam)} & \leq C_{eq} |||u-\ue|||^2 \\
& = C_{eq} \mcA(u-\ue,u-\ue) \\
& \leq C_{eq} k(\epsilon) \|\ue\|_{L^2(\Omega\cup\Gam)} \|u-\ue\|_{L^2(\Omega\cup\Gam)},
\end{aligned}
\end{displaymath}
where $k(\epsilon)$ is obtained from the constants in \eqref{eq:bilin-form-estimate}. We finally conclude that
\begin{equation}\label{eq:eps-error}
\|u-\ue\|_{L^2(\Omg\cup\Gam)} \leq C_{eq} k(\epsilon) \|\ue\|_{L^2(\Omega\cup\Gam)},
\end{equation}
where the constant $k(\epsilon)$ is such that $k(\epsilon)\to 0$ as $\epsilon\to 0$.

%%%%%%%%%%%%%%%%%%%%%%%%%%%%%%%%%%%%%%%%%%%%%%%%%%%%
%%%%%%%%%%%%%%%%%%%%%%%%%%%%%%%%%%%%%%%%%%%%%%%%%%%%
\section{Finite element formulation}\label{sec:fem}
In this section we introduce a FE discretization of problem \eqref{eq:weak-homo}, highlight the associated computational challenges, and describe how the formulation introduced in the previous section helps circumventing them. 
Let $\mcT_h$ be a shape-regular triangulation of $\Omega\cup\Gamma$ into $N_L$ finite elements $\{\mcE_l\}_{l=1}^{N_L}$; the latter $\mcE_l$ can either be triangles and/or quadrilaterals in two dimensions and tetrahedra and/or hexaheadra in three dimensions\footnote{See \cite{DEliaFEM2020} for a description of appropriate triangulation techniques for nonlocal problems.}. The parameter $h$ represents the size of the triangulation and corresponds to the larger element diameter. Also, let $V_0^{N_h}$ be a finite dimensional subspace of $V_0$ of dimension $N_h$, proportional to $h^{-1}$, and let $\{\varphi_i\}_{i=1}^{N_h}$ be a basis for $V_0^{N_h}$. In this work we consider Lagrange basis functions over the triangulation $\mcT_h$. Thus, we can write the FE solution $w_h$ of equation \eqref{eq:weak-homo} as $w_h(\xb) = \sum_{i=1}^{N_h}W_i \varphi_i(\xb)$. By using this expression and $\varphi\in\{\varphi_i\}_{i=1}^{N_h}$, equation \eqref{eq:weak-homo} reduces to the algebraic system
\begin{equation}
A {\mathbf W} = {\mathbf F},
\end{equation}
where ${\mathbf W}\in\mbR^{N_h}$ is the vector whose components are the degrees of freedom of the numerical solution $w_h$, $\mathbf F$ is such that ${\mathbf F}_i= \widetilde \mcF(\varphi_i)$, and $A$ is the stiffness matrix with entries
\begin{equation}\label{eq:Aij}
A_{ij} = \mcA(\varphi_i,\varphi_j)=
\iint\limits_{(\Omega\cup\Gamma)^2}  \left(\varphi_i(\xb)-\varphi_i(\yb)\right)\left(\varphi_j(\xb)-\varphi_j(\yb)\right)\gamma(\xb,\yb)  d\yb\,d\xb.
\end{equation}

%%%%%%%%%%%%%%%%%%%%%%%%%%%%%%%%%%%%%%%%%%%%%%%%%%%%%%%%%%%%%
\subsection{Circumventing computational challenges}
The computation of the entries of the stiffness matrix $A$ raises several diverse challenges. In this work we specifically focus on the challenges related to the presence of the indicator function in the definition of the kernel. Other challenges, such as the presence of singularities in fractional-type kernels or peridynamics kernels are not considered here. We point out that our method can be combined with any technique that takes into account the presence of the singularity. In fact, while the singularity is located at the center of the ball, the issues considered in this paper arise at the boundary. To make our description clear, we rewrite \eqref{eq:Aij} by explicitly indicating the domain of integration, i.e.
\begin{equation}\label{eq:Aij-support}
\begin{aligned}
A_{ij} & =
C_\delta \int\limits_{\Omega\cup\Gamma} 
\int\limits_{(\Omega\cup\Gamma)\cap B_\delta(\xb)} 
\left(\varphi_i(\xb)-\varphi_i(\yb)\right)\left(\varphi_j(\xb)-\varphi_j(\yb)\right)\eta(\xb,\yb)  d\yb\,d\xb\\
& = C_\delta \sum\limits_{l=1}^{N_L} \sum\limits_{k=1}^{N_L} 
\,\intl_{\mcE_l} \intl_{\mcE_k\cap B_\delta(\xb)}
\left(\varphi_i(\xb)-\varphi_i(\yb)\right)\left(\varphi_j(\xb)-\varphi_j(\yb)\right)\eta(\xb,\yb)  d\yb\,d\xb,
\end{aligned}
\end{equation}
where we split the integrals over the elements with the purpose of using composite quadrature rules. In fact, global quadrature rules used, e.g., over the ball $B_\delta(\xb)$ for the inner integration are not convenient due to the basis functions' bounded support \cite{DEliaFEM2020}. 

It is evident that the first challenge that one has to face is the integration over partial elements: when the element $\mcE_k$ is not fully contained in the ball, standard quadrature rules such as Gauss quadrature rules defined over $\mcE_k$ are not suitable due to the presence of the discontinuity induced by the indicator function. Thus, it is necessary to determine the intersection regions $\mcE_k\cap B_\delta(\xb)$ and define quadrature rules there. This task, while affordable in two dimensions, becomes extremely complex and impractical in three dimensions. Furthermore, when $B_\delta(\xb)$ is a Euclidean ball, such regions are curved so that appropriate approximations or quadrature rules for curved domains must be taken into account \cite{DEliaFEM2020}. A key observation is that these issues do not arise in case of smooth kernel functions, e.g. functions that do not abruptly jump to zero outside of $B_\delta(\xb)$, but that approach zero smoothly. This would allow the use of quadrature rules defined over the whole element $\mcE_k$, circumventing the issue of determining intersections or integrating over curved regions. 

The parametrized kernel introduced in Section \ref{sec:weak-form} is such that the transition to zero happens smoothly (as an example, for constant kernel functions $\eta$, the kernel function is a piece-wise polynomial in $C^4$. 
Thus, the inner integration can be performed over the whole element $\mcE_k$, using accurate enough quadrature rules, without worrying about the presence of a discontinuity. We then propose to solve the approximate, parametrized problem
\begin{equation}
\mcA_\epsilon(w_{h,\epsilon},\varphi_i) = \widetilde F(\varphi_i), 
\quad \forall\,i=1,\ldots N_h,
\end{equation}
for which the entries of the stiffness matrix, that, with an abuse of notation, we still denote by $A$, read
\begin{equation}\label{eq:Aij-support2}
\begin{aligned}
A_{ij} & =
C_{\delta,\epsilon} \sum\limits_{l=1}^{N_L} \sum\limits_{k=1}^{N_L} 
\,\intl_{\mcE_l} \intl_{\mcE_k}
\left(\varphi_i(\xb)-\varphi_i(\yb)\right)\left(\varphi_j(\xb)-\varphi_j(\yb)\right)\eta(\xb,\yb)\mude(\xb,\yb)  d\yb\,d\xb.
\end{aligned}
\end{equation}
By avoiding the problem of determining intersecting elements, this approach makes three-dimensional implementation a much simpler task. 

\begin{remark}
The convergence of the solution $w_{h,\epsilon}$ to the continuous solution $w$ depends on both the discretization parameter $h$ and the mollifying parameter $\epsilon$. An adaptive quadrature procedure, introduced in the following section, will further contribute to the overall approximation error, as we discuss and illustrate in Section \ref{sec:numerics}.
\end{remark}

%%%%%%%%%%%%%%%%%%%%%%%%%%%%%%%%%%%%%%%%%%%%%%%%%%%%%%%%%%%%%
%%%%%%%%%%%%%%%%%%%%%%%%%%%%%%%%%%%%%%%%%%%%%%%%%%%%%%%%%%%%%
\section{Adaptive quadrature rules}\label{sec:adaptivity}
As already pointed out, the use of the mollifier, in place of the characteristic function, removes the difficulty of integrating discontinuous functions. The transition region of the mollifier has thickness $2 \varepsilon$ and it is important to choose quadrature rules that can appropriately capture this region, especially if $\varepsilon \ll h$. Note that the presence of the transition region affects the regularity of both the inner and the outer integrands. As we explain below, only one adaptive rule is necessary, applied to either the outer or inner integral. A quadrature rule with few points can be fast but is also inaccurate, one with many points can be accurate but is also expensive, especially in higher dimensions. To this end, adaptive quadrature rules have been proven to be accurate and efficient \cite{mousavi2012efficient}. The advantage of using an adaptive scheme is that only the portion of the element overlapping with the transition region needs to be recursively refined, considerably reducing the computational time. Moreover,  for fixed $\epsilon$, each partitioning has the effect of halving the ratio $h/\varepsilon$. Thus, it is always possible to chose a number of adaptive refinements such that $h \approx \varepsilon$ and for which a quadrature rule with few points is accurate enough. Note that the adaptive quadrature rule devised here is not standard, because the refinement criterion is controlled by the distance between points in the outer and inner integrals. The details of the algorithm are given below.

We recall that we denote by $A$ the stiffness matrix corresponding to the parametrized bilinear form $\mcA_\epsilon$. It is convenient to rewrite its entries as $A_{ij} = A^{11}_{ij}+A^{12}_{ij}+A^{21}_{ij}+A^{22}_{ij}$, where each term is given by
\begin{align}
A^{11}_{ij} &= \int\limits_{\Omega\cup\Gamma} \int\limits_{\Omega\cup\Gamma} \gamma_{\epsilon}(\xb,\yb)\,\varphi_i(\xb) \varphi_j(\xb) d\yb\,d\xb,\label{J11}\\
A^{12}_{ij} &=-\int\limits_{\Omega\cup\Gamma} \int\limits_{\Omega\cup\Gamma} \gamma_{\epsilon}(\xb,\yb)\,\varphi_i(\xb) \varphi_j(\yb) d\yb\,d\xb,\label{J12}\\
A^{21}_{ij} &=-\int\limits_{\Omega\cup\Gamma} \int\limits_{\Omega\cup\Gamma} \gamma_{\epsilon}(\xb,\yb)\,\varphi_i(\yb) \varphi_j(\xb) d\yb\,d\xb,\label{J21}\\
A^{22}_{ij} &=\int\limits_{\Omega\cup\Gamma} \int\limits_{\Omega\cup\Gamma} \gamma_{\epsilon}(\xb,\yb)\,\varphi_i(\yb) \varphi_j(\yb) d\yb\,d\xb.\label{J22}
\end{align}

\smallskip
The following proposition allows us to express $A_{ij}$ only as a sum of two of the terms above, as we show in Corollary \ref{cor:Aequality}.
\begin{proposition} \label{equalityOfJij}
Let $f_1$ and $f_2 \in V$, and let $g(\xb,\yb)$ be a symmetric function, then 
\begin{align}
&\int\limits_{\Omega\cup\Gamma} \int\limits_{\Omega\cup\Gamma} g(\xb,\yb)\,f_1(\xb) f_2(\yb) d\yb\,d\xb =\int\limits_{\Omega\cup\Gamma} \int\limits_{\Omega\cup\Gamma} g(\xb,\yb)\,f_1(\yb) f_2(\xb) d\yb\,d\xb.
\end{align}
\end{proposition}

\begin{proof}
\begin{align}
&\int\limits_{\Omega\cup\Gamma} \int\limits_{\Omega\cup\Gamma} g(\xb,\yb)\,f_1(\xb) f_2(\yb) d\yb\,d\xb & \hspace{2cm}  \,\nonumber\\
=& \int\limits_{\Omega\cup\Gamma} \int\limits_{\Omega\cup\Gamma} g(\xb,\yb)\,f_1(\xb) f_2(\yb) d\xb\,d\yb & \mbox{reversing the order}\nonumber\\
=& \int\limits_{\Omega\cup\Gamma} \int\limits_{\Omega\cup\Gamma} g(\xb,\yb)\,f_1(\yb) f_2(\xb) d\yb\,d\xb &
\begin{array}{r}
\mbox{renaming variables} \\
\mbox{and using the symmetry of $g$}
\end{array}\nonumber
\end{align}
\end{proof}
\begin{corollary}\label{cor:Aequality}
The entries of the stiffness matrix $A$ satisfy the following equality
$$A_{ij} = 2 A^{11}_{ij} + 2 A^{12}_{ij} = 2 A^{21}_{ij} + 2 A^{22}_{ij}$$. 
\end{corollary}
\begin{proof}
The proof follows from the definition of $A_{ij}$ and Proposition \ref{equalityOfJij}. Namely, $A^{11}_{ij}=A^{22}_{ij} \mbox{ and } A^{12}_{ij}=A^{21}_{ij}.$
\end{proof}

Let $Q_1=\{(\xb_{q_1},w_{q_1})\}_{q_1}$ and $Q_2=\{(\xb_{q_2},w_{q_2})\}_{q_2}$ denote sets of quadrature points and associated weights representing  two different {\it composite quadrature} rules for the numerical integration over the region $\Omega \cup \Gamma$. To preserve the equality between $A^{11}_{ij}$ and $A^{22}_{ij}$, and between $A^{12}_{ij}$ and $A^{21}_{ij}$ we
use $Q_1$ and $Q_2$ to numerically evaluate the inner and outer integrals as follows:
\begin{align}
&& A_{ij}^{11} = \sum_{q_1 \in Q_1} \sum_{q_2 \in Q_2} \gamma_{\epsilon}(\xb_{q_1},\xb_{q_2})\,\varphi_i(\xb_{q_1}) \varphi_j(\xb_{q_1}) w_{q_2}{w_{q_1}}, \label{J11N}\\
&& A_{ij}^{12} = \sum_{q_1 \in Q_1} \sum_{q_2 \in Q_2} \gamma_{\epsilon}(\xb_{q_1},\xb_{q_2})\,\varphi_i(\xb_{q_1}) \varphi_j(\xb_{q_2}) w_{q_2}{w_{q_1}},\label{J12N}\\
&& A_{ij}^{21} = \sum_{q_2 \in Q_2} \sum_{q_1 \in Q_1} \gamma_{\epsilon}(\xb_{q_2},\xb_{q_1})\,\varphi_i(\xb_{q_1}) \varphi_j(\xb_{q_2}) w_{q_1}{w_{q_2}},\label{J21N}\\
&& A_{ij}^{22} = \sum_{q_2 \in Q_2} \sum_{q_1 \in Q_1} \gamma_{\epsilon}(\xb_{q_2},\xb_{q_1})\,\varphi_i(\xb_{q_1}) \varphi_j(\xb_{q_1}) w_{q_1}{w_{q_2}}\label{J22N}.
\end{align}

According to Corollary \ref{cor:Aequality}, only two terms among the ones above need to be computed. We choose to compute $A_{ij}^{21}$ and $A_{ij}^{22}$ in \eqref{J21N} and \eqref{J22N}. We adopt an adaptive scheme for the outer quadrature and Gaussian composite quadrature rules for the inner integration over the elements that intersect the ball. 
These choices are empirical, i.e. they have been guided by numerical experiments that showed that for $A_{ij}^{21}$ and $A_{ij}^{22}$ using an adaptive quadrature for the outer integral is more efficient than using it for the inner one.

%We adopt a standard numerical integration for one of the two quadrature rules (i.e. Gaussian quadrature) and an adaptive scheme for the other one. Our numerical simulations show that the use of  adaptivity for the selection of the quadrature points of $Q_2$ is far superior than using adaptivity for $Q_1$. Moreover, since $A_{ij}^{11} = A_{ij}^{22}$ and $A_{ij}^{12} = A_{ij}^{21}$, only the two terms $A_{ij}^{21}$ and $A_{ij}^{22}$ in \eqref{J21N} and \eqref{J22N} are computed, thus employing the adaptive scheme only for the outer integral. 
%Because we are using a finite element discretization, the domain $\Omega \cup \Gamma$ can be partitioned as $\Omega \cup \Gamma = \bigcup_{l=1}^{N_{L}} \overline{\mathcal{E}}_l$, where $\mathcal{E}_l$ is the $l$-th finite element and $N_{L}$ is the total number of finite elements.
For every $\mathcal{E}_l\in\mcT_h$ we define
\begin{align}\label{k_l}
    \mathcal{K}_l = \{ m\in\{1,\ldots,N_L\}:  \| \xb-\yb\|_{\ell_\infty} \ge \delta+ \varepsilon, \,\,  \forall \xb \in \mathcal{E}_l,  \forall \yb \in \mathcal{E}_m\,\,\},
\end{align}
and $\mathcal{J}_l = \mathcal{K}^c_l$, i.e. the complement of $\mathcal{K}_l$ in $\{1,\ldots,N_L\}$.
For $k=1,2$ and $l=1,\ldots,N_L$ let $Q_k^l=\{(\xb_{q_{k_l}},w_{q_{k_l}})\}_{q_{k_l}} \subset Q_{k}$ denote the subset of $Q_k$ composed of those quadrature points and weights obtained by only considering the quadrature points (and associated weights) that lie within element $\mathcal{E}_l$.  
Then, integrals \eqref{J21N} and \eqref{J22N} can be rewritten as
\begin{align}
&& A_{ij}^{21} = \sum_{l=1}^{N_L} \sum_{q_{2_l} \in Q_2^l } \sum_{m \in \mathcal{J}_l} \sum_{q_{1_m} \in Q_1^m} \gamma_{\epsilon}(\xb_{q_{2_l}},\xb_{q_{1_m}})\,\varphi^{m}_i(\xb_{q_{1_m}}) \varphi^{l}_j(\xb_{q_{2_l}}) w_{q_{1_m}}{w_{q_{2_l}}},\label{J21Nr}\\
&& A_{ij}^{22} = \sum_{l=1}^{N_L} \sum_{q_{2_l} \in Q_2^l } \sum_{m \in \mathcal{J}_l} \sum_{q_{1_m} \in Q_1^m} \gamma_{\epsilon}(\xb_{q_{2_l}},\xb_{q_{1_m}})\,\varphi^{m}_i(\xb_{q_{1_m}}) \varphi^{l}_j(\xb_{q_{1_m}}) w_{q_{1_m}}{w_{q_{2_l}}}\label{J22Nr}.    
\end{align}
Note that in the equations above the terms in the sum are nonzero only when 
the support of the basis functions intersects the elements.
An adaptive quadrature with midpoint refinement is adopted for $Q_2^l$. 
The pseudo-code that describes the adaptivity algorithm is reported in Algorithm \ref{pseudo}: we employ recursive calls with input arguments $L_{min}$, $L_{max}$, $L_{cur}$, $\mathcal{E}_l$, and $\mathcal{J}_l$. In each call $L_{min}$ and $L_{max}$ are fixed parameters, and represent the minimum and the maximum level of refinement, with $L_{max} \ge L_{min} \ge 1$. $L_{cur}$ is the current level of refinement.
In the initial call $\mathcal{E}_l$ and $\mathcal{J}_l$ are the ones defined above and $L_{cur} =1$, 
while in the recursive calls these three arguments are subject to changes as described below.

\begin{itemize}
\item If $L_{cur} < L_{min}$, then $\mathcal{E}_l$ is split in $2^N$ sub-elements $\mathcal{E}_{l_i}$ using the midpoint rule and for each of them the adaptive integration function is called again increasing $L_{cur}$ by one and using the same index set $\mathcal{J}_l$.
\item If $L_{cur} = L_{max}$ integration is performed on $\mathcal{E}_{l}$ for the outer integral and on each element indexed by $\mathcal{J}_l$ for the inner integral, without any further refinement for $\mathcal{E}_{l}$. The numerical integration is performed using standard Gauss Legendre quadrature rules both for the outer and inner integrals.

\item If $L_{min} \le L_{cur} < L_{max}$, from the index set $\mathcal{J}_l$ two new index sets
$\mathcal{J}_l^{int}$ and $\mathcal{J}_l^{ref}$ are extracted, for which $\mathcal{E}_{l}$ is either integrated or further refined.
For any $m$ in $\mathcal{J}_l$, the maximum distance from $\mathcal{E}_{m}$ to $\mathcal{E}_{l}$ is computed. If this distance is less than $\delta - \varepsilon$, then $m$ is added to $\mathcal{J}_l^{int}$, otherwise the minimum distance from $\mathcal{E}_{m}$ to $\mathcal{E}_{l}$ is computed. If this distance is less than $\delta + \varepsilon$ then $m$ is added to the index set $\mathcal{J}_l^{ref}$.
If $\mathcal{J}_l^{int}$ is non-empty, integration is performed on $\mathcal{E}_{l}$ for the outer integral and on each element indexed by $\mathcal{J}_l^{int}$ for the inner integral. 
If $\mathcal{J}_l^{ref}$ is non-empty, then $\mathcal{E}_{l}$ is split in $2^N$ sub-elements $\mathcal{E}_{l_i}$ and for each of them the adaptive integration function is called again increasing $L_{cur}$ by one and using $\mathcal{J}_l^{ref}$ as index set. 
\end{itemize}

\begin{algorithm}
\label{pseudo}
\begin{algorithmic}
  \Function {Adaptive Integration}{$L_{min}$, $L_{max}$, $L_{cur}$, $\mathcal{E}_l$, $\mathcal{J}_l$}
    \If {$L_{cur}<L_{min}$}
      \State {split $\mathcal{E}_l$ into $2^N$ new sub-elements $\mathcal{E}_{l_i}$}
      \For{$i=1,\dots,2^N$ }
        \State{ \sc{Adaptive Integration}($L_{min}$, $L_{max}$, $L_{cur} + 1$, $\mathcal{E}_{l_i}$, $\mathcal{J}_l$)}
      \EndFor
    \ElsIf {$L_{cur} =  L_{max}$}
      \State{\sc{Integration}($\mathcal{E}_l$, $\mathcal{J}_l$)}
    \Else
      \State{$\mathcal{J}^{int}_l = \varnothing$}
      \State{$\mathcal{J}^{ref}_l = \varnothing$}
      \ForAll{$m \in \mathcal{J}_l$}
        \If{$\mbox{max dist}(\mathcal{E}_l,\mathcal{E}_m) < \delta - \varepsilon$}
          \State{add $m$ to the index set $\mathcal{J}^{int}_l $}
        \ElsIf{$\mbox{min dist}(\mathcal{E}_l,\mathcal{E}_m) < \delta + \varepsilon$}
          \State{add $m$ to the index set $\mathcal{J}^{ref}_l$}
        \EndIf
      \EndFor
      \If{$\mathcal{J}^{int}_l \ne \varnothing$}
        \State{\sc{Integration}($\mathcal{E}_l,$ $\mathcal{J}^{int}_l)$}
      \EndIf
      \If{$\mathcal{J}^{ref}_l \ne \varnothing$}
        \State{split $\mathcal{E}_l$ into $2^N$ new sub-elements $\mathcal{E}_{l_i}$}
        \For{$i=1,\dots,2^N$ }
          \State{\sc{Adaptive Integration}($L_{min}$, $L_{max}$, $L_{cur} + 1$, $\mathcal{E}_{l_i}$, $\mathcal{J}_l^{ref}$)}
        \EndFor
      \EndIf
    \EndIf
  \EndFunction
\end{algorithmic}
\end{algorithm}

\subsection{Approximate maximum and minimum distances between elements \label{AppDistSec}}
Evaluating the exact distances between two elements can be computationally expensive, especially for unstructured three-dimensional meshes. Hence, in practice, we use conservative distances that are simple to compute in place of the maximum and the minimum. Namely, we first loop  over the nodes of each element
to find the minimum and maximum coordinates in each dimension, denote them by $\xb_{\min}$, $\xb_{\max}$ and by $\yb_{\min}$, $\yb_{\max}$. Here $\xb_{\min}$ and $\xb_{\max}$ are the vectors containing the minimum and maximum coordinates of the bounding box containing $\mathcal{E}_l$. Similarly, $\yb_{\min}$ and $\yb_{\max}$ are the vectors containing the minimum and maximum coordinates of the bounding box containing $\mathcal{E}_m$.
For each dimension $k=1,\ldots,N$ evaluate the two quantities $d_1^k = x^k_{\min} - y^k_{\max}$ and $d_2^k = y^k_{\min} - x^k_{\max}$. Finally, we approximate the maximum and minimum distances with the two quantities
\begin{align}
& \mbox{aprx max dist}(\mathcal{E}_l,\mathcal{E}_m) = \sqrt{ \sum_{k=1,\dots,N} \max \left( {d_1^k}^2,\, {d_2^k}^2 \right) },\label{maxdist}\\
& \mbox{aprx min dist}(\mathcal{E}_l,\mathcal{E}_m) = \max_{k=1,\dots,N} \left( \max \left( 0, d_1^k,\, d_2^k\right) \right) . \label{mindist}
\end{align}

The approximate maximum distance in \eqref{maxdist} is the maximum among the distances between opposite vertices of the two bounding boxes.
For example in two dimensions it would be one among the distances $\|SW - NE\|_{\ell_2}$, $\|SE - NW\|_{\ell_2}$, $\|NE - SW\|_{\ell_2}$ and $\|NW - SE\|_{\ell_2}$, with S meaning South, N meaning North and so on. This is true regardless of the reciprocal position of the two boxes.
To better understand \eqref{mindist}, consider the projections of the bounding boxes in the direction of $k$. Recall that, for fixed $k$, at least one between $d_1^k$ and $d_2^k$ is always negative. The other is positive only if the 2 projections do not overlap. In this case, 
$\max \left( 0, d_1^k,\, d_2^k\right)$ is the minimum (positive) distance between 
the 2 non-overlapping projections. Finally, we approximate the minimum distance with the largest projected distance.

It is easy to see that
\begin{align}
& \mbox{aprx max dist}(\mathcal{E}_l,\mathcal{E}_m)  \ge \mbox{max dist}(\mathcal{E}_l,\mathcal{E}_m) ,\label{maxdistin}\\
& \mbox{aprx min dist}(\mathcal{E}_l,\mathcal{E}_m)  \le \mbox{min dist}(\mathcal{E}_l,\mathcal{E}_m), \label{mindistin}
\end{align}
where the sign of the inequalities assures the conservative approach in the adaptive integration algorithm. For a kernel whose support is identified by a ball in the topology defined by the $L^2$ norm, the minimum distance could be also approximated by
$$\mbox{aprx min dist}(\mathcal{E}_l,\mathcal{E}_m) = \sqrt{ \sum_{k=1,\dots,N} \max \left(0, d_1^k,\, d_2^k\right)^2}.$$ 
This last would give a sharper inequality in \eqref{mindistin}, however it would not work for the case where the support of the kernel is an $L^{\infty}$ ball. In this work, for generality, we have chosen to always use formula \eqref{mindist}. Note that \eqref{mindist} is also used in \eqref{k_l} in place of $\| \cdot \|_{l_{\infty}}$ to identify the elements indexed by  $\mathcal{J}_l$.

%%%%%%%%%%%%%%%%%%%%%%%%%%%%%%%%%%%%%%%%%%%%%%%%%%%%%%%%%%%%%
%%%%%%%%%%%%%%%%%%%%%%%%%%%%%%%%%%%%%%%%%%%%%%%%%%%%%%%%%%%%%
\section{Numerical results}\label{sec:numerics}
In this section, we present the results of numerical tests for FE
discretizations of two-dimensional ($n=2$) and three-dimensional ($n=3$) problems. These results allow us to illustrate the theoretical results presented in the previous sections and highlight the efficiency of our approach.

We first show the consistency of the proposed method; specifically, fixing the mesh and letting the maximum level of adaptive refinement $L_{max}$ increase, we study the behavior of the discretization error with respect to an analytic solution that belongs to the FE space. 
Then, we investigate the convergence of discretized solutions to the continuous one as the mesh is refined. 
To better understand the behavior of the algorithm and the specific sources of error, we devised a specific numerical test to isolate the error induced by the presence of the mollifier and analyze the convergence behavior with respect to $\varepsilon$. Accuracy comparisons with the algorithm proposed in the paper by D'Elia et al. \cite{DEliaFEM2020} are also provided.

Due to the intrinsically high computational costs of nonlocal simulations, a parallel implementation of the algorithm is proposed and its scalability properties are analyzed both in two and three dimensions.

The two-dimensional tests are carried out on quadrilateral, triangular, and mixed meshes, i.e. meshes consisting of both quadrilateral and triangular elements. The three-dimensional simulations are carried out on a hexahedral mesh.

We consider constant kernels supported on Euclidean balls of radius $\delta$. To guarantee the consistency of the nonlocal diffusion operator with the classical Laplacian for polynomials up to degree three and its convergence to the classical Laplacian as $\delta\to 0$, we select the constants $C_\delta$ and $C_{\delta,\varepsilon}$ in \eqref{eq:gamma} and \eqref{eq:approx-kernel}, respectively, as follows
$$
\begin{aligned}
n=2:\quad & 
C_{\delta} = \dfrac{4 \kappa} {\pi  \delta^4}\,, &
C_{\delta, \varepsilon} = \dfrac{C_{\delta}}{1+\dfrac{6}{11}\left(\dfrac{\varepsilon}{\delta}\right)^2 + \dfrac{3}{143}\left(\dfrac{\varepsilon}{\delta}\right)^4},\\
n=3:\quad & 
C_{\delta} = \dfrac{15 \kappa} {4 \pi  \delta^5}\,, &
C_{\delta,\varepsilon} = \dfrac{C_{\delta}}{1+\dfrac{10}{11}\left(\dfrac{\varepsilon}{\delta}\right)^2 + \dfrac{15}{143}\left(\dfrac{\varepsilon}{\delta}\right)^4}.
\end{aligned}
$$
For sufficiently smooth $u$, these choices guarantee quadratic convergence of the nonlocal Laplacian to the local one for $\delta\rightarrow 0$:
$$ \mcL u(\xb) = \Delta u(\xb) + O(\delta^2)
\quad \hbox{and} \quad 
\mcL_\varepsilon u(\xb) = \Delta u(\xb) + O(\delta^2).
$$

%%%%%%%%%%
\paragraph{Sources of numerical error}
Let $u_h$ be the FE solution; the numerical error can heuristically be split into three separate contributions, i.e.
\begin{equation}\label{errorSum}
%E(\varepsilon, h, L_{\max}) = 
\|u-u_h\|_{L^2(\Omega\cup\Gamma)} \leq C_1 h^{p_1} + C_2 \varepsilon^{p_2} + E_i(\varepsilon, h, L_{max}),   
\end{equation}
where, $h$ is the characteristic size of the mesh, $C_1$ and $C_2$ are positive constants independent of $h$, $\varepsilon$ and $L_{max}$, and $p_1$ and $p_2$ are positive integers that represent the rates of convergence with respect to $h$ and $\varepsilon$, respectively. The first term on the right-hand side is the interpolation error and depends on the FE family used to discretize the problem\footnote{For convergence rates of FE discretizations in presence of integrable kernels, we refer the reader to \cite{du12}.}. The second term is the error induced by the presence of the mollifier and, in the $L^2$ metric, it is bounded by the expression in \eqref{eq:eps-error}. The last term, $E_i$, is the numerical integration error, that
for fixed external and internal Gauss quadrature rules depends on $\varepsilon$, $h$ and $L_{max}$.

For $h\rightarrow0$ or $L_{max}\rightarrow \infty$ the integration error $E_i$ vanishes, whereas for 
$\varepsilon \rightarrow 0$ it increases, since the transition to zero of the mollifier features higher gradients. In all simulations we empirically set the integration parameters so that the integration error $E_i$ could be negligible compared to $C_1h^{p_1}+C_2 \varepsilon^{p_2}$. In such a context, $\varepsilon$ cannot be selected independently of $h$ or $L_{max}$; explicit dependence is provided in each simulation.

%%%%%%%%%%%%%%%%%%%%%%%%%%%%%%%%%%%%%%%%%%%%%%%%%%%%%%%%%
\subsection{Two-dimensional tests}
\begin{figure}[t!]
   \centering
   \includegraphics[height=1.3in]{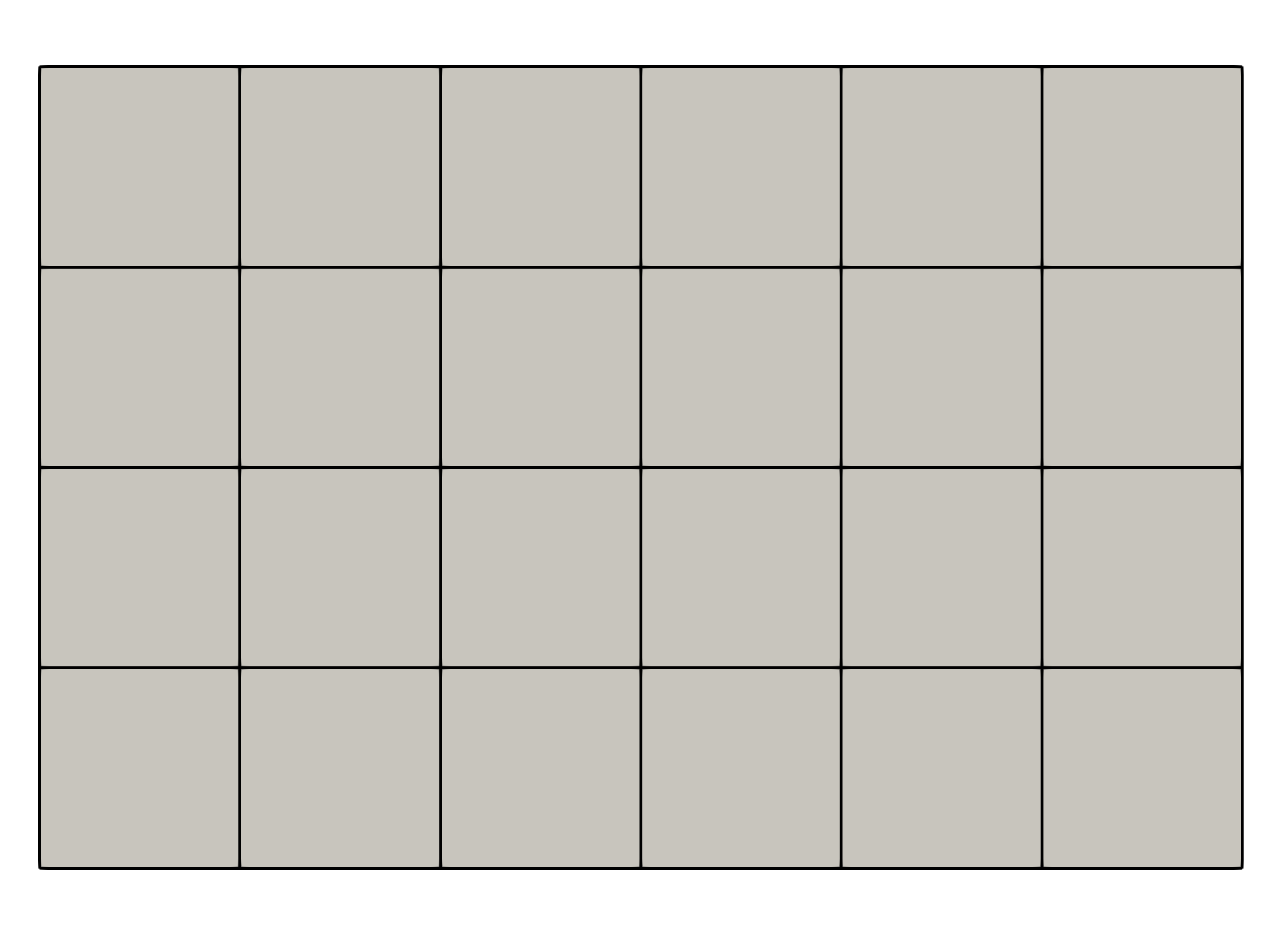}
   \quad
   \includegraphics[height=1.3in]{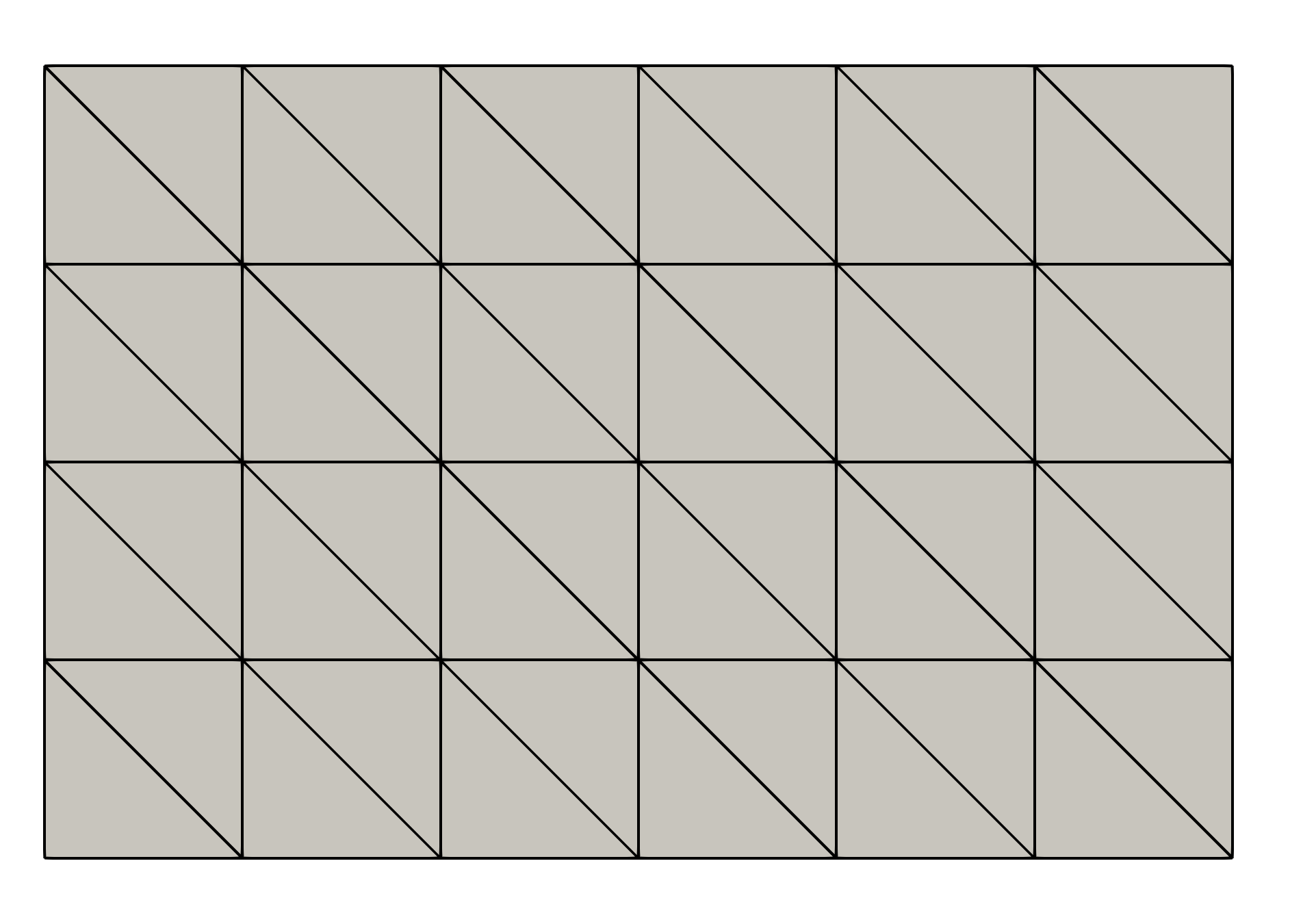}
   \quad
   \includegraphics[height=1.3in]{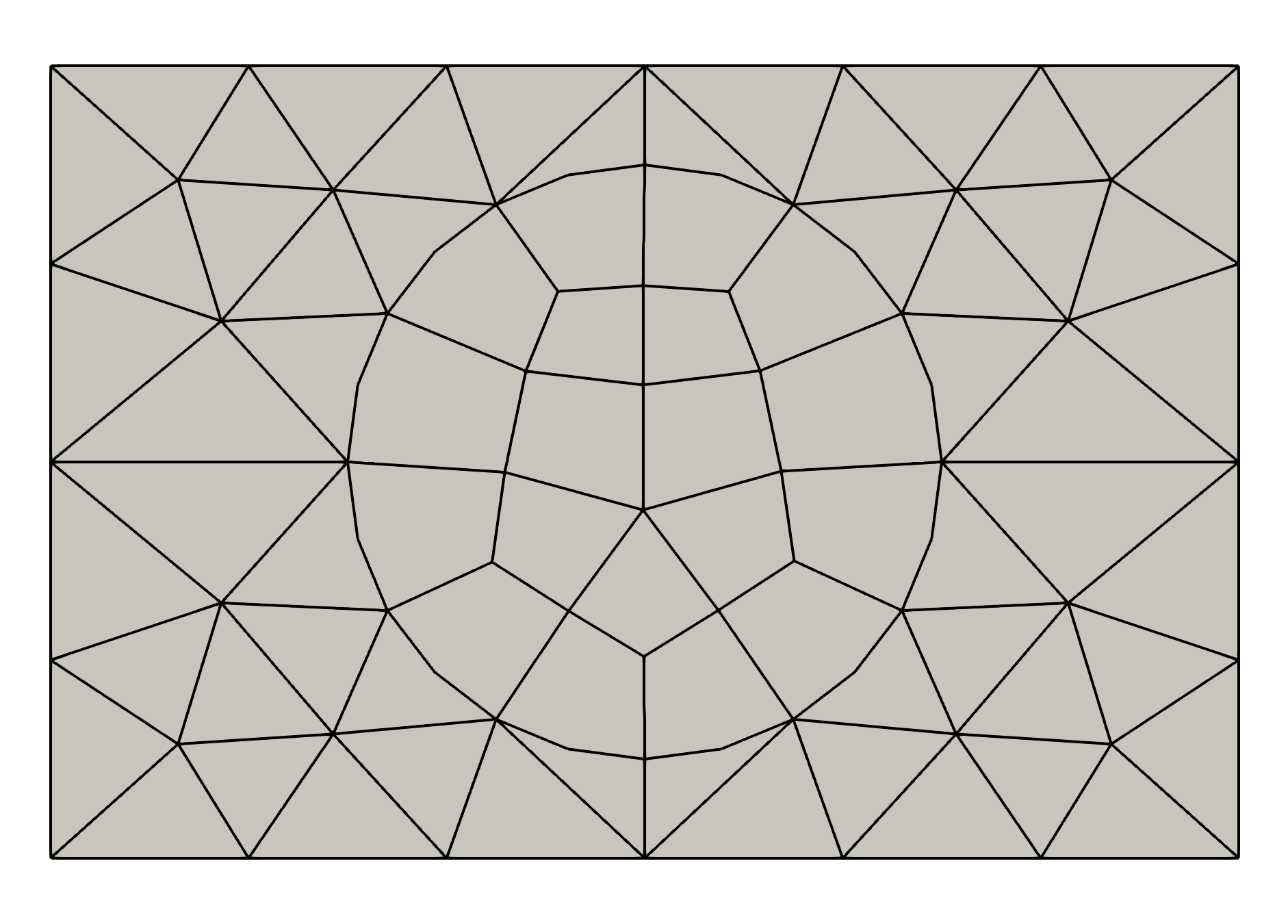}
\caption{Meshes used for the two-dimensional numerical simulations: quadrilateral (left, number of degrees of freedom with linear FE discretization $N_h = 35$, and $N_h = 93$ for quadratic FE discretization), triangular (center, $N_h = 35$ for linear FE, $N_h = 117$ for quadratic FE) and mixed (right, $N_h = 58$ for linear FE, $N_h = 190$ for quadratic FE).}
   \label{fig_domain}
\end{figure}
Two-dimensional numerical simulations of nonlocal operators can be found in several works in literature, see, e.g., \cite{Macek2007}, \cite{Wang2014}, and \cite{Vollman2019}. However, such studies are often designed for structured mesh only. On the other hand, our method can be applied to any type of mesh. 
We consider the domain $\Omega = [-0.6, 0.6] \times [-0.4, 0.4]$ and three different meshes, see Figure \ref{fig_domain} for a coarse example.
Linear and quadratic Lagrange FE spaces are considered.
Unless otherwise stated, we use Gauss-Legendre $3\times3$ product rule for quadrilateral elements and Dunavant 7-point rule for triangles.
In all tests we consider the error with respect to an analytic, manufactured solution, $u(\xb)$. 
For all $\xb \in \Gamma$, the nonlocal Dirichlet volume constraint is set to $g(\xb) =u(\xb)$ and the forcing term is known analytically as $f(\xb)=-\mcL u(\xb)$.
%$$f(\xb) = -2 \int_{\Omega\cap\Gamma} \gamma(\xb,\yb)\, (u(\yb) - u(\xb)) \,d\yb.$$
%This last in general differs from $$f_{\varepsilon}(\xb) =-2 \int_{\Omega\cap\Gamma} \gamma_{\varepsilon}(\xb,\yb)\, (u(\yb) - u(\xb)) \,d\yb.$$

%%%%%%%%%%%
\paragraph{Consistency}
We choose $u\in V^{N_h}$, so that the interpolation error contribution in \eqref{errorSum} is identically zero, regardless of $h$. We compute the discretization error $\|u - u_h\|_{L^2}$ for a fixed mesh while increasing the values of the maximum level of adaptivity $L_{max}$.

For the linear manufactured solution $u=1+x+y$ the forcing term is $f(\xb)=0$.
%$$f(\xb)=f_{\varepsilon}(\xb) = 0.$$
In this case the numerical errors are always zero for both the linear and the quadratic FE spaces, for any $L_{max}$ and $\varepsilon$. Although welcome, this result is an over-achievement, as it is obtained only because the forcing terms are zero. Thus, it should not be taken as a reference.
For the quadratic manufactured solution $u=x^2+y^2$ the forcing term is $f(\xb)=-2$. %$$f(\xb)=f_{\varepsilon}(\xb) = -2.$$ 
In this case we consider only the quadratic FE space, because it is the only one that can reproduce exactly the solution.
We choose $\varepsilon$ as 
$$\varepsilon = \varepsilon_0 \left(\frac{3}{4}\right)^{L_{max}-3}$$
so that $\varepsilon\rightarrow 0 $ for $L_{max} \rightarrow \infty$, i.e. for $L_{max}\to\infty$ both the mollifier and the integration errors in \eqref{errorSum} vanish.

In Table \ref{tab_consistency1} we report the errors of the numerical tests for the 
quadrilateral (QUAD) and triangular (TRI) meshes after one refinement, together with all the values of the parameters $h$, $\delta$, $\varepsilon_0$ and $L_{\min}$. The $L^2$-norm of the error decreases down to machine precision, 
when $L_{max}$ is increased, for both quadrilateral and triangular meshes, illustrating the consistency of the implemented adaptivity method in two dimensions.

\begin{center}
\begin{table}[t]%
\centering
\caption{Consistency test: errors $\|u - u_h\|_{L^2}$ as $L_{max}$ increases for $u=x^2+y^2$, $h=0.1$, $\delta=0.2$, $\varepsilon_0=0.0125$, $L_{min} = 1$, $\varepsilon = \varepsilon_0 (3/4)^{L_{max}-3}$ and
 quadratic FE. \label{tab_consistency1}}
 \begin{tabular*}{150pt}{@{\extracolsep\fill}c|cc@{\extracolsep\fill}}
 %\multicolumn{3}{c}{   $\|u - u_h\|_{L^2}$}  \vspace{+0.3cm}\\ 
\toprule
$\mathbf{L_{max}}$ & \textbf{QUAD}     &   \textbf{TRI} \\
\midrule
        3   & 8.123E-05 &  2.713E-05 \\
        4   & 1.278E-05 &  2.052E-06\\
        5   & 3.453E-07 &  1.193E-07\\
        6   & 1.793E-09 &  1.774E-09 \\
        7   & 7.073E-10 &  5.479E-11 \\
\bottomrule
\end{tabular*}
\end{table}
\end{center}

\paragraph{$\boldsymbol h$-Convergence}
We consider the convergence with the respect to the grid size $h$ on quadrilateral, triangular and mixed meshes. We use the manufactured solution 
$u(\xb) = x^3 + y^3$
for  which  the  corresponding source term is given by $f(\xb) = -\mcL u(\xb)
.=
- \Delta u(\xb) = - 6 (x + y)$ for $\xb \in \Omega$.
We analyze the convergence of the finite FE adaptive nonlocal solution
$u_h$ to the analytic solution $u$ as we progressively halve the mesh size $h$ by operating on a parameter referred to as $ml$, as shown in \eqref{ml_eq}. We consider fixed $L_{max}$, $L_{min}$, and $\delta$,
whereas $h$ and $\varepsilon$ depend on $ml$ (mesh level) as follows 
\begin{align}\label{ml_eq}
h = h_0 \left(\frac{1}{2}\right)^{ml-2} 
\quad \hbox{and} \quad
\varepsilon =
\varepsilon_0\left(\frac{2}{3}\right)^{ml-2}.
\end{align}
In Table \ref{tab_mesh_convergence1} we report the values of $\|u - u_h\|_{L^2}$
and the corresponding rate of convergence $p$  as $ml$ grows, evaluated in full awareness with respect to $h$ only, with the approximate formula
\begin{equation}
\label{eq:convergence_h}
p \approxeq \ln\big({E(h)}/{E\left({h}/{2}\right)} \big)/\ln(2) \,.
\end{equation}
For linear FE discretization we obtain quadratic convergence. 
This is optimal since it resembles the optimal convergence rate $p_1 = 2$ 
of the interpolation error \cite{du12}. Namely, it indicates that in \eqref{errorSum} the mollifier and the integration errors are negligible with respect to the interpolation error.
Instead, in case of quadratic FE discretization, the observed rate is $p \approx 3$ (similar to the optimal one $p_1 = 3$) only for $ml \leq 4$, but it deteriorates for higher values of $ml$, i.e. as we refine the meshes (this behavior happens consistently on all the tested meshes, quadrilateral, triangular and mixed).
This is due to the combined effect of the mollifier and integration errors that
start dominating for $h \rightarrow 0$. 
Since $\varepsilon \propto (2/3)^{ml-2}$, the mollifying function defined in 
\eqref{eq:mollifier} exhibits a sharper gradient as we increase $ml$ inducing a less accurate numerical integration,  and, hence, higher values of $E_i$. Moreover, since the mesh refinement significantly reduces the interpolation error $C_1 h^{p_i}$, as we increase $ml$ the error contibution $C_2 \varepsilon^{p_2}$ becomes dominant affecting the overall convergence rate $p$.

We further test the convergence rate with respect to $h$ considering the fourth-degree polynomial
\begin{equation}
\label{eq:quartic}
u(\xb) = x^4 + y^4\,,
\end{equation}
for  which  the  corresponding  source  term is given by $f(\xb) = - 12 (x^2 + y^2) -  \delta^2$ 
for $\xb \in \Omega$. Similarly to the previous test, in Table \ref{tab_mesh_convergence2} we
report the numerical results for linear and quadratic FE discretizations, on quadrilateral, triangular and mixed meshes. Again, $p \approx 2$ 
for all the linear discretizations, while $p \approx 3$ for quadratic discretizations only for $ml \leq 3$. Same considerations as for the previous test can be inferred.

\begin{center}
\begin{table}[t]%
\centering
\caption{$h$-Convergence test: errors $\|u - u_h\|_{L^2}$ and computed order $p$  as the mesh level $ml$ increases for $u=x^3 + y^3$, $L_{min} = 1$, $L_{max} = 3$, $\delta=0.2$, $h_0 = 0.1$, $h = h_0 (1/2)^{ml-2}$, $\varepsilon_0=0.0125$, $\varepsilon = \varepsilon_0(2/3)^{ml-2}$.}
 \label{tab_mesh_convergence1}
 \begin{tabular*}{350pt}{@{\extracolsep\fill}c |c c |c c | c c@{\extracolsep\fill}}
 %\multicolumn{7}{c}{   $\|u - u_h\|_{L^2}$ and $p$}  \vspace{+0.3cm}\\ 
        \toprule
     & \multicolumn{2}{c|}{\textbf{QUAD}} &  \multicolumn{2}{c|}{\textbf{TRI}} &  \multicolumn{2}{c}{\textbf{MIXED}}\\ 
        \midrule
       $ml$ &     
       bilinear & quadratic &  linear & quadratic &  linear & quadratic \\
       \midrule
        2 & 4.363E-03 & 6.077E-05 & 4.373E-03 & 6.389E-05 & 2.386E-03 & 3.352E-05 \\
        & \textbf{1.996} & \textbf{3.090} & \textbf{1.999} & \textbf{3.015}& \textbf{2.028} & \textbf{3.188} \\
        3 & 1.094E-03 & 7.135E-06 & 1.094E-03 & 7.906E-06 & 5.848E-04 & 3.677E-06 \\
        & \textbf{1.998} & \textbf{2.995} & \textbf{1.999} & \textbf{2.989} & \textbf{2.015} & \textbf{2.865}\\
        4 & 2.738E-04 & 8.950E-07 & 2.737E-04 & 9.956E-07 & 1.447E-04 & 5.046E-07 \\
        & \textbf{2.000} & \textbf{2.690} & \textbf{1.999} & \textbf{2.694} & \textbf{2.006} & \textbf{1.674}\\
        5 & 6.845E-05 & 1.387E-07 & 6.845E-05 & 1.538E-07 & 3.602E-05 & 1.581E-07\\
        & \textbf{2.000} & \textbf{1.012} & \textbf{2.000} & \textbf{1.315} & \textbf{2.003} & \textbf{0.880}\\
        6 & 1.711E-05 & 6.878E-08 & 1.711E-05 & 6.179E-08 & 8.983E-06 & 8.589E-08\\
\bottomrule
\end{tabular*}
\end{table}
\end{center}

\begin{center}
\begin{table}[t]%
\centering
\caption{$h$-Convergence test: errors $\|u - u_h\|_{L^2}$ and computed order $p$  as the mesh level $ml$ increases for $u=x^4 + y^4$, $L_{min} = 1$, $L_{max} = 3$, $\delta=0.2$, $h_0 = 0.1$, $h = h_0 (1/2)^{ml-2}$, $\varepsilon_0=0.0125$, $\varepsilon = \varepsilon_0 (2/3)^{ml-2}$.}
  \label{tab_mesh_convergence2}
 \begin{tabular*}{350pt}{@{\extracolsep\fill}c |c c |c c | c c @{\extracolsep\fill}}
 %\multicolumn{9}{c}{  $\|u - u_h\|_{L^2}$ and $p$}  \vspace{+0.3cm}\\ 
        \toprule
     & \multicolumn{2}{c|}{\textbf{QUAD}} &  \multicolumn{2}{c|}{\textbf{TRI}} &  \multicolumn{2}{c}{\textbf{MIXED}} \\ 
        \midrule
       $ml$ &     
       bilinear & quadratic &  linear & quadratic &  linear & quadratic  \\
       \midrule
        2 & 6.383E-03 & 9.544E-05 & 6.398E-03 & 1.027E-04 & 3.267E-03 & 5.561E-05\\
        & \textbf{1.972}	& \textbf{2.892}	& \textbf{1.976}	& \textbf{2.914}	& \textbf{2.005}	& \textbf{2.769}  \\
        3 & 1.627E-03 & 1.285E-05 & 1.626E-03 & 1.362E-05 & 8.138E-04 & 8.155E-06 \\
        & \textbf{1.986}	& \textbf{2.161}	& \textbf{1.985}	& \textbf{2.240}	& \textbf{2.000}	& \textbf{1.677}\\
        4 & 4.106E-04 & 2.872E-06 & 4.105E-04 & 2.883E-06 & 2.034E-04 & 2.550E-06 \\
        & \textbf{1.990}	& \textbf{1.374}	& \textbf{1.990}	& \textbf{1.373}	& \textbf{1.994}	& \textbf{1.214}\\
        5 & 1.033E-04 & 1.108E-06 & 1.033E-04 & 1.113E-06 & 5.106E-05 & 1.099E-06 \\
        & \textbf{1.990}	& \textbf{1.178}	& \textbf{1.990}	& \textbf{1.176}	& \textbf{1.988}	& \textbf{1.159}\\
        6 & 2.599E-05 & 4.895E-07 & 2.599E-05 & 4.923E-07 & 1.287E-05 & 4.920E-07 \\
\bottomrule
\end{tabular*}
\end{table}
\end{center}

\paragraph{$\boldsymbol \varepsilon$-Convergence}
We analyze the contribution of the mollifier to the discretization error, i.e. $C_2 \varepsilon^{p_2}$ in \eqref{errorSum}.
We consider the fourth-order polynomial in
\eqref{eq:quartic} on a triangular mesh with a quadratic FE discretization. In Table \ref{tab_eps_conv}, for fixed $\delta$ and $L_{min}$, we report on the error $\|u - u_h\|_{L^2}$ and the computed convergence rate $p_2$, starting from $\varepsilon = 0.1$ and progressively halving it.
To minimize the interpolation error $C_i h^{p_1}$ and the integration error $E_i$, so that the overall error is dominated by the $\varepsilon$-contribution, for each $\varepsilon$, during the tests the number of adaptive refinements $L_{max}$ and of the mesh level $ml$ have been increased with a brute force procedure until the values of the overall error did not change significantly anymore. The error values reported in the table are those obtained only after this steady state was reached. The interpolation error (IE) in the table is reported as a reference and is obtained by numerically solving the local counterpart of the nonlocal Poisson problem at the finer mesh level, i.e. the one to which is associated the steady state.  Also as a reference, we report the relative error (RE) between the interpolation error and the overall error
$$
\hbox{RE} =\frac{ \mbox{IE} } {\|u-u_h\|_{L^2(\Omega\cup\Gamma)}}\,.
$$
The smaller the value of RE the more accurate the data, since the impact of the interpolation error on the overall error vanishes. 
The $\varepsilon$-convergence order is deliberately computed as
\begin{equation}
\label{eq:convergence_p2}
p_2 \approxeq \ln\big({E(\varepsilon)}/{E\left({\varepsilon}/{2}\right)} \big)/\ln(2) \,,
\end{equation}
and is approximately $p_2 \approx 2$ for all the tested mollifier thicknesses.
%%%%%%%%%%%%%%%%%%%%%%%%%%%%%%%%%%%
\begin{center}
\begin{table}[t]%
\centering
\caption{
%\gc{forse dovremmo mettere la colonna TRI alla fine dato che e' la piu rilevante?}\eugenio{ Io lascerei quest'ordine e definirei ER come reciproco di quello che e' ora, diventerebbe l'errore relatico di IE sull'~errore totale}
$ \varepsilon$-Convergence test: $\|u - u_h\|_{L^2}$ and computed order $p_2$ as $\varepsilon$ decreases for $u=x^4+y^4$, $\delta=0.2$, $L_{min} = 1$, and quadratic FE.
\label{tab_eps_conv}}
 \begin{tabular*}{200pt}{@{\extracolsep\fill}c|c|c|c@{\extracolsep\fill}}
% \multicolumn{4}{c}{   $\|u - u_h\|_{L^2}$ and $p_2$}  \vspace{+0.3cm}\\ 
\toprule
$\mathbf{\varepsilon}$ & \textbf{IE} & \textbf{TRI}    & \textbf{RE}   \\
\midrule
0.1      &   1.602E-06  & 7.107E-04 & 2.25E-03  \\
         &     & \textbf{2.09} &   \\
0.05     &   1.602E-06  & 1.671E-04 & 9.58E-03 \\
         &     &  \textbf{1.90} &   \\
0.025    &   1.602E-06  & 4.4716E-05 & 3.58E-02 \\
         &     &  \textbf{1.91} &       \\
0.0125   &   2.002E-07  & 1.189E-05 & 1.68E-02 \\
         &     &  \textbf{1.95} &    \\
0.00625  &   2.002E-07  & 3.073E-06 & 6.51E-02 \\
        &     &  \textbf{1.92} &    \\
0.00313  &   2.002E-07  & 8.117E-07 & 2.46E-01 \\
%         &     &  \textbf{.} &    \textbf{.}\\
%0.003125 &   .  & . & . \\
\bottomrule
\end{tabular*}
\end{table}
\end{center}
%%%%%%%%%%%%%%%%%%%%%%%%%%%%%%%%%%%
\paragraph{Comparison with other algorithms}
In Section \ref{AppDistSec} we introduced the maximum and minimum distance between
elements in order to determine all elements intersecting the ball that identifies the support of the kernel function, for which integration is performed. However, in the literature, there are other techniques to determine whether an element should be considered during the assembly of the FE matrix. As an example, the paper by D'Elia et al.\cite{DEliaFEM2020} proposes a technique for which 
only the elements whose barycenter lies inside the ball are considered for integration. This approach results in an approximation of the ball by a union of whole finite elements, and shows second order $h$-convergence for all FE spaces (linear, quadratic, etc.), as the rate is determined by the ball approximation. 
With the purpose of testing the performance of our method against current approaches, we conduct a comparison study, where we consider two different values of $\varepsilon$ and compare our results with those obtained with a barycenter approach. Specifically, we use
$\varepsilon_0 = 0.0125$, with $\varepsilon = \varepsilon_0 (2/3)^{ml-2}$ 
and $\varepsilon_0 = 0.0250$, with $\varepsilon = \varepsilon_0 (1/2)^{ml-2}$ (note that the $\varepsilon$ decrease differently). 
For the the adaptive quadrature rule presented in this paper we use a Legendre quadrature 
rule both for the internal and external integral, whereas for the barycenter method we use 
a hybrid Lobatto $\times$ Legendre quadrature rule,  
Lobatto in the outer integral and Legendre in the inner one.
In Table \ref{tab_mesh_comparison} we report the discretization errors for linear FE on a quadrilateral
mesh as the mesh level $ml$ increases.
The $h$-convergence rate $p$ is computed as in \eqref{eq:convergence_h}.
We see that this rate is always optimal ($p \approx 2$) regardless of the value of $\varepsilon$; in other words, the mollifier and numerical quadrature contributions are negligible and the error is dominated by the FE interpolation contribution. 
On the other hand, the barycenter method shows bigger errors and less regular convergence order, due to geometric error introduced by the ball approximation. The computational times are all comparable since $L_{max}$ has been taken equal to 1, and all  quadrature rules have the same number of points.
The interpolation error IE reported in the table has the same meaning as already discussed.

\begin{center}
\begin{table}[t]%
\centering
\caption{Comparison test: errors $\|u - u_h\|_{L^2}$ and $h$-convergence rate $p$ with a quadrilateral mesh and linear FE as the mesh refinement level $ml$ increases for $u=x^4 + y^4$, $\delta=0.2$, $h_0 = 0.1$, $h = h_0 (1/2)^{ml-2}$. For the adaptive algorithm: (Legendre $3\times3$) $\times$ (Legendre $3\times3$), $L_{min} = 1$, $L_{max} = 3$. For the barycenter algorithm: (Lobatto $3\times3$) $\times$ (Legendre $3\times3$)}
\label{tab_mesh_comparison}
 \begin{tabular*}{350pt}{@{\extracolsep\fill}c | c | c c | c @{\extracolsep\fill}}
%\multicolumn{5}{c}{  $\|u - u_h\|_{L^2}$ and $p$}  \vspace{+0.3cm}\\ 
 \toprule
& \textbf{IE} & \multicolumn{2}{c|}{\textbf{Adaptive}}
& \textbf{Barycenter} \\
\midrule
$ml$ & - & $\varepsilon = 0.0125\,(2/3)^{ml-2}$ &
$\varepsilon = 0.025\,(1/2)^{ml-2}$ &  \\
\midrule
2 & 3.280E-03 &  3.281E-03 &  3.290E-03 &  5.428E-03 \\
& \textbf{1.99} & \textbf{2.02} & \textbf{2.01} & \textbf{1.95} \\
3 & 8.250E-04 & 8.092E-04 & 8.147E-04 & 1.401E-03 \\
& \textbf{2.00} & \textbf{1.99} & \textbf{2.00} & \textbf{1.82} \\
4 & 2.066E-04 & 2.034E-04 & 2.038E-04 & 3.974E-04 \\
& \textbf{2.00} & \textbf{1.99} & \textbf{2.00} & \textbf{2.00} \\
5 & 5.168E-05 & 5.118E-05 & 5.104E-05 & 9.909E-05 \\
& \textbf{2.01} & \textbf{1.99} & \textbf{2.00} & \textbf{1.69} \\
6 & 1.287E-05 & 1.287E-05 & 1.273E-05 & 3.073E-05 \\
& \textbf{1.99} & \textbf{1.98} & \textbf{2.00} & \textbf{1.67} \\
7 & 3.231E-06 & 3.270E-06 & 3.195E-06 & 9.645E-06 \\
\bottomrule
\end{tabular*}
\end{table}
\end{center}

%%%%%%%%%%%%%%%%%%%%%%%%
\subsection{Three-dimensional tests}
\begin{figure}[t!]
   \centering
   \includegraphics[height=1.6in]{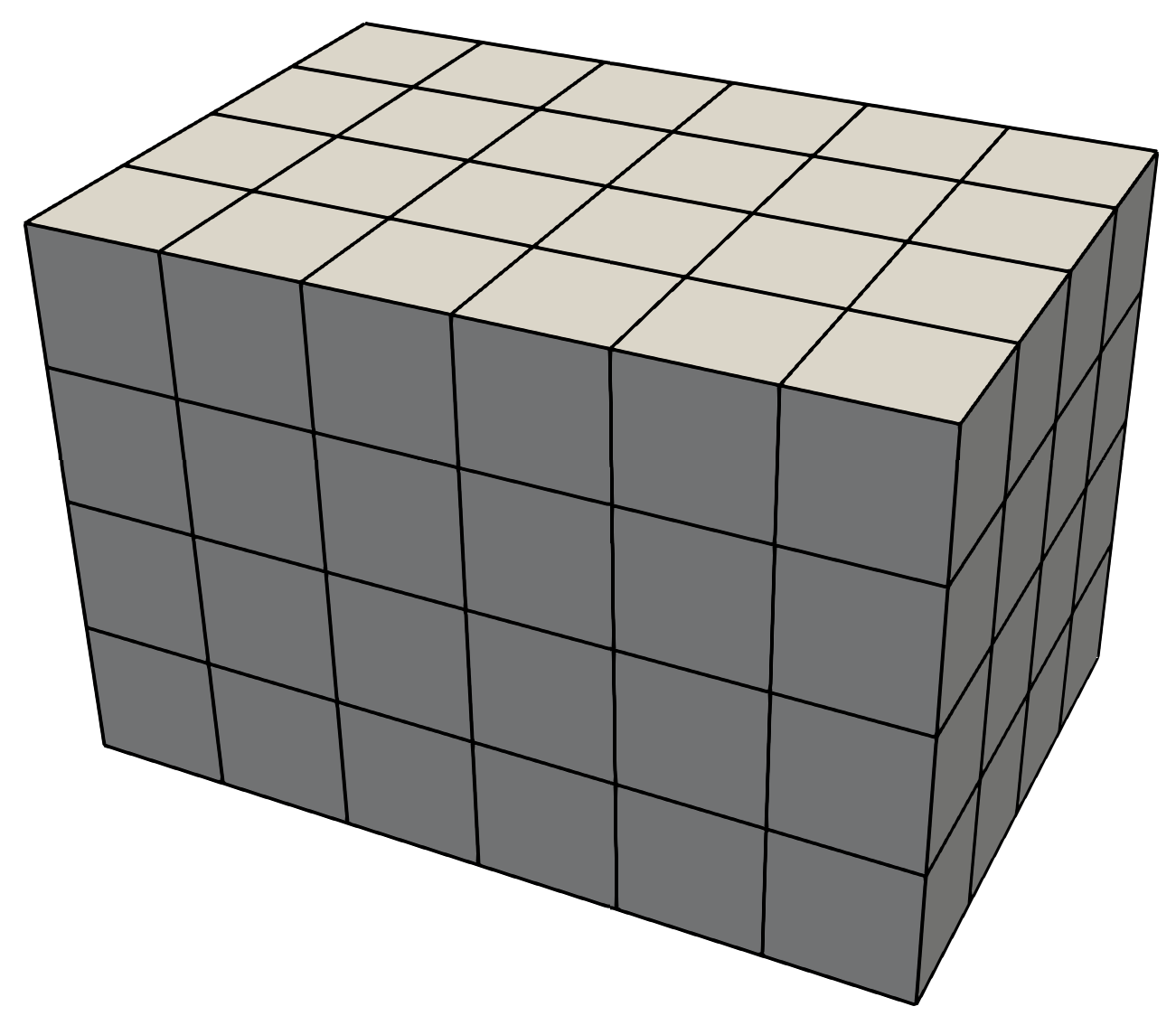}
\caption{Hexahedron mesh used for the three-dimensional numerical simulations, with 
$N_h = 175$ (linear FE discretization) and $N_h = 605$ (quadratic FE discretization). }
   \label{fig_domain_3d}
\end{figure}
%%%%%%%%%%%%%%%%%%%%%%%%
%%%%%%%%%%%%%%%%%%%%%%%%
\begin{center}
\begin{table}[!ht]%
\centering
 \caption{Consistency in 3D: $h=0.1$, $\delta=0.2$, $\varepsilon_0=0.0125$, $u=x^2+y^2+z^2$,
 $L_{min} = 1$, $\varepsilon = \varepsilon_0 \left(\frac{3}{4}\right)^{L_{max}-2}$,
 quadratic FE.}
 \label{tab_consistency_3D}
 \begin{tabular*}{75pt}{@{\extracolsep\fill}c|c@{\extracolsep\fill}}
 \multicolumn{2}{c}{   $\|u - u_h\|_{L^2}$}  \vspace{+0.3cm}\\ 
\toprule
     $\mathbf{L_{max}}$ & \textbf{HEX}\\%  & \textbf{MIXED} \\
        \midrule
        2   &   6.863E-04 \\%  & 4.917E-04  \\
        3   &   2.499E-05 \\% & 2.6928E-04 \\
        4   &   3.241E-06  \\%& 1.5195E-04  \\
        5   &   7.035E-08   \\
\bottomrule
\end{tabular*}
\end{table}
\end{center}
%%%%%%%%%%%%%%%%%%%%%%%%
Three-dimensional simulations of nonlocal problems are incredibly challenging, especially in a variational setting, due to the prohibitively high computational 
effort. At the time of this study, the nonlocal literature still does not offer efficient, scalable algorithms for three-dimensional FE implementations of nonlocal solvers for kernels with bounded support. Efficient algorithms for fractional operators are proposed in \cite{AinsworthGlusa2017}, whereas implementations for compactly-supported, integrable kernels on structured grids can be found in \cite{Wang2014} and \cite{Vollman2019}. Here, not only do we introduce an efficient three-dimensional implementation, but we also propose a scalable, parallel implementation.  

We proceed as in the two-dimensional case. In all our tests we consider the numerical domain 
$\Omega = [-0.6, 0.6] \times [-0.4, 0.4] \times [-0.4, 0.4]$, discretized with the regular hexahedron mesh reported in Figure \ref{fig_domain_3d}.

We first consider a consistency test; for fixed values of $\delta$,
$L_{min}$ and $h$ we increase the refinement level $L_{max}$. For the same reasons explained in the previous section, we consider the analytic solution $u = x^2 + y^2 + z^2$ for which the corresponding forcing term is given by $f(\xb)=-6$ for $\xb\in\Omega$.
Results for quadratic FE on the structured hexahedral mesh are reported in Table \ref{tab_consistency_3D}.
As expected, the error decreases when the adaptive refinement level $L_{max}$ is increased, illustrating the consistency of the adaptive algorithm in three dimensions.

Next, we test the $h$-convergence on the structured hexahedral mesh for both linear and quadratic FE. For the analytic solution $u(\xb) = x^3 + y^3 + z^3$ and fixed $L_{min}$, $L_{max}$, and $\delta$ we study the 
behavior of the error $\|u - u_h\|_{L^2}$ as we halve $h$. Here, the  corresponding  source  term is $f(\xb) = - \Delta u(\xb) =
- \mcL u = - 6 (x + y + z)$ for $\xb \in \Omega$.
In Table \ref{tab_convergence3D} (left), we report on the errors and the computed convergence rate $p$. 
We also consider the analytic solution
$u(\xb) = x^4 + y^4 + z^4$
with  the  corresponding  source  term 
$f(\xb) = - 12 (x^2 + y^2 + z^2) - 6/7 \delta^2$ for $\xb \in \Omega$. 
For the same parameters as for the cubic polynomial, results of numerical experiments are reported in Table \ref{tab_convergence3D} (right).
We observe an optimal convergence rate $p \approx 2$ for linear FE discretizations for both the cubic and the quartic cases. 
Similarly to the two dimensional case, the convergence rate for the quadratic FE discretizations starts from $p \approx 3$ and it deteriorates when $ml$ grows. In general, the three-dimensional numerical results are consistent with the two-dimensional ones and the same considerations can be inferred. 

\begin{center}
\begin{table}[t]%
\centering
\caption{$h$-convergence: errors $\|u - u_h\|_{L^2}$ and convergence rate $p$ as the mesh level $ml$ increases for $u=x^3 + y^3 + z^3$ (left) and $u=x^4 + y^4 + z^4$ (right), $L_{min} = 1$, $L_{max} = 2$, $\delta=0.2$, $h_0 = 0.2$,
 $h = h_0 (1/2)^{ml-1}$, $\varepsilon_0=0.01875$, 
 $\varepsilon = \varepsilon_0 (2/3)^{ml-1}$.}
 \label{tab_convergence3D}
 \begin{tabular*}{130pt}{@{\extracolsep\fill}c|cc@{\extracolsep\fill}}
 %\multicolumn{3}{c}{$\|u - u_h\|_{L^2}$ and $p$}  \vspace{+0.3cm}\\ 
\toprule
                & \multicolumn{2}{c}{   \textbf{HEX}} \\
\midrule
     $ml$ & linear & quadratic    \\
\midrule
        1   &   2.220E-02   &   6.739E-04    \\
        & \textbf{1.988}	& \textbf{2.972} \\
        2   &   5.593E-03   &   8.587E-05    \\
        & \textbf{1.998}	& \textbf{2.797} \\
        3   &   1.400E-03   &   1.235E-05    \\
        & \textbf{1.999}	& \textbf{2.816} \\
        4   &   3.501E-04   &   1.753E-06    \\
\bottomrule
\end{tabular*}
\hspace{2cm}
\begin{tabular*}{130pt}{@{\extracolsep\fill}c|cc@{\extracolsep\fill}}
% \multicolumn{3}{c}{$\|u - u_h\|_{L^2}$ and $p$}  \vspace{+0.3cm}\\ 
\toprule
                & \multicolumn{2}{c}{   \textbf{HEX}} \\
\midrule
     $ml$ & linear & quadratic    \\
\midrule
        1   &   3.299E-02   &   9.328E-04     \\
        & \textbf{1.964}	& \textbf{2.864} \\
        2   &   8.452E-03   &   1.281E-04     \\
        & \textbf{1.985}	& \textbf{2.674} \\
        3   &   2.134E-03   &   2.006E-05     \\
        & \textbf{1.992}	& \textbf{2.303} \\
        4   &   5.361E-04   &   4.064E-06   \\ 
\bottomrule
\end{tabular*}
\end{table}
\end{center}

\begin{comment}
\begin{center}
\begin{table}[t]%
\centering
\caption{$h$-Convergence: errors $\|u - u_h\|_{L^2}$ and computed rate $p$ as the mesh level $ml$ increases for $u=x^4 + y^4 + z^4$, $L_{min} = 1$, $L_{max} = 2$, $\delta=0.2$, $h_0 = 0.2$,
 $h = h_0 (1/2)^{ml-1}$, $\varepsilon_0=0.01875$, 
 $\varepsilon = \varepsilon_0 (2/3)^{ml-1}$.}
 \label{tab_convergence_3D_quartic}
 \begin{tabular*}{130pt}{@{\extracolsep\fill}c|cc@{\extracolsep\fill}}
% \multicolumn{3}{c}{$\|u - u_h\|_{L^2}$ and $p$}  \vspace{+0.3cm}\\ 
\toprule
                & \multicolumn{2}{c}{   \textbf{HEX}} \\
\midrule
     $ml$ & linear & quadratic    \\
\midrule
        1   &   3.299E-02   &   9.328E-04     \\
        & \textbf{1.964}	& \textbf{2.864} \\
        2   &   8.452E-03   &   1.281E-04     \\
        & \textbf{1.985}	& \textbf{2.674} \\
        3   &   2.134E-03   &   2.006E-05     \\
        & \textbf{1.992}	& \textbf{2.303} \\
        4   &   5.361E-04   &   4.064E-06       \\
        5   &               &   \\
\bottomrule
\end{tabular*}
\end{table}
\end{center}
\end{comment}

%%%%%%%%%%%%%%%%%%%%%%%%%%%%%%%%%%%%%%%%%%%%%%%%%%%
\subsection{Parallel implementation}
The high computational effort necessary to perform three-dimensional nonlocal simulations requires a parallel implementation, even for relative small-sized problems. This is mostly due to the fact that, unlike what happens in the local settings, the bandwidth of a nonlocal matrix increases with mesh refinement, since the radius, $\delta$, of the nonlocal neighborhood remains fixed. In this section, we first highlight the need of a parallel implementation, we then propose a parallel algorithm, and finally illustrate its efficiency on two- and three-dimensional problems.

%%%%%%%%%%%
\paragraph{Complexity of nonlocal simulations}
For each refinement, i.e. every time the mesh size is halved, the average number of elements contained in the kernel's support increases of a factor $2^N$. Consequently, the number of unknowns increases of the same factor and the number of nonzero entries in the matrix increases of $2^{2N}$, as opposed to local settings where the average increase at every mesh refinement is $2^N$. 
To better understand the impact of such increase let us consider the quadrilateral mesh used in the previous simulations for the quadratic solution, with a sparse matrix where the allocation for a non-zero entry is 12-bytes: 4 to specify the column location in the row (int) and 8 to store the value (double).  
At mesh level $ml=6$, the number of unknowns is $N_h=148353$ and, for $\delta  = 0.2$, the maximum number of entries in one row is 34749. The memory allocation for this mesh level requires approximately 80 GB of memory, while a corresponding local problem would require no more than 50 MB of memory.

Similarly, for each mesh refinement, the computational time to assemble the nonlocal matrix increases in average by the same factor $2^{2N}$, since the number of elements increases of $2^N$ and the average 
number of elements contained inside the kernel's support increases by $2^N$. For a local problem the assembly time increases only of a factor $2^N$.
These differences in memory allocation and CPU time indicate that a parallel implementation of the nonlocal assembly is vital to make nonlocal models a preferred and viable modeling option.

%%%%%%%%%%%
\paragraph{Details on the implementation}
We begin our description by stressing the fact that, even in a local context, a parallel implementation of a FE algorithm is nontrivial. In what follows, we assume that the reader is familiar with the FE method and relatively accustomed to its parallel implementation \cite{babuvska1990some}; thus, we proceed by highlighting the main challenges that arise in nonlocal implementations.

Our parallel algorithm has been implemented in FEMuS \cite{FEMuS} and is publicly available on GitHub. FEMuS is a in-house FE C++ library that interfaces with PETSc \cite{balay2019petsc}, which provides the linear algebra library for the parallel solver. The parallalization of the nonlocal assembly has been entirely developed within FEMuS.

As it is common in parallel FE settings, the mesh elements are partitioned among the processes $N_p$;  FEMuS uses the METIS/PARAMETIS \cite{karypis1997parmetis}
library for partitioning unstructured meshes. For each process $I$, we let $\Omega_I$ be the domain composed of the elements owned by $I$ that overlap with $\Omega$ and $\Pi_I$ be the domain composed of the elements owned by $I$ that overlap with $\Gamma$. Although highly desirable, each sub-domain $\Omega_I \cup \Pi_i$ does not need to be simply connected. Similar to the definition of $\Gamma$, we define the interaction domain $\Gamma_I$ of $\Omega_I$ as
\begin{equation}\label{eq:interaction-domain_I}
\Gamma_I = \{\yb\in (\Omega \cup \Gamma) \setminus \Omega_I:  |\xb-\yb|\leq\delta + \varepsilon\; \text{for some} \; \xb\in\Omega_I\cup\Pi_I\}.
\end{equation}

While in the local case two processes $I$ and $J$ have to exchange information only between elements on a shared boundary, i.e. on $\partial\Omega_I \cap \partial\Omega_J$, in the nonlocal case the two processes need to exchange information whenever an element of $I$ intersects the interaction domain $\Gamma_J$ or vice-versa. We denote by $\Gamma_{JI}$ the region made up by all the elements owned by process $J$ that intersect with the interaction domain $\Gamma_I$, for $I,J = 1, \cdots, N_p$.
Set $\widetilde{\Gamma}_I = \cup_{J\in Np} \Gamma_{JI}$. The following relations hold
\begin{align*}
& \Gamma_{JI} \cap \Gamma_{KI} = \varnothing \; \mbox{ for } J \ne K, \\
& {\Gamma}_I \subseteq \widetilde{\Gamma}_I, \\
&\Gamma_{II} = \Pi_I, \\
&\Gamma = \cup_{I\in Np} \Gamma_{II}, \\
&\Gamma_{II} \cap \Gamma_{JJ} = \varnothing \; \mbox{ for } I \ne J.
\end{align*}

The following expression clarifies what operations can be performed within one process and which ones require exchange of information among processes. We consider a general integral, whose numerical computation is ubiquitous in Algorithm \ref{pseudo}. For any function $w(\xb,\yb)$, any double integral such as those defined in the first part of the paper can be rewritten as
\begin{align}
&\iint\limits_{(\Omega\cup\Gamma)^2}  \gamma_{\epsilon}(\xb,\yb)\, w(\xb,\yb) d\yb\,d\xb\nonumber\\
&=\sum_{I=1}^{N_p}\;\int\limits_{\Omega\cup\Gamma}\; \int\limits_{\Omega_I\cup\Gamma_{II}} \gamma_{\epsilon}(\xb,\yb)\,w(\xb,\yb) d\yb\,d\xb \nonumber\\
&=\sum_{I=1}^{N_p}\;\int\limits_{\Omega_I\cup\Gamma_I}\; \int\limits_{\Omega_I\cup\Gamma_{II}} \gamma_{\epsilon}(\xb,\yb)\,w(\xb,\yb) d\yb\,d\xb\nonumber\\
&=\sum_{I=1}^{N_p}\;\int\limits_{\Omega_I\cup \widetilde{\Gamma}_I} \int\limits_{\Omega_I\cup\Gamma_{II}} \gamma_{\epsilon}(\xb,\yb)\,w(\xb,\yb) d\yb\,d\xb\nonumber\\
&=\sum_{I=1}^{N_p}\;\int\limits_{\Omega_I\cup \left( \cup_{J\in Np} \Gamma_{JI} \right)  }\; \int\limits_{\;\,\Omega_I\cup\Gamma_{II}} \gamma_{\epsilon}(\xb,\yb)\,w(\xb,\yb) d\yb\,d\xb\nonumber\\
& = 
\sum_{I=1}^{N_p}\left(\quad\iint\limits_{(\Omega_{I} \cup \Gamma_{II})^2} \gamma_{\epsilon}(\xb,\yb)\,w(\xb,\yb) d\yb\,d\xb +
\sum_{\underset{J \ne I}{J=1}}^{N_p} \;\int\limits_{\Gamma_{JI}} \int\limits_{\;\,\Omega_{I} \cup \Gamma_{II}}  \gamma_{\epsilon}(\xb,\yb)\,w(\xb,\yb) d\yb\,d\xb\right) \nonumber\\
& = 
\sum_{I=1}^{N_p}\left(\quad\iint\limits_{(\Omega_{I} \cup \Gamma_{II})^2} \gamma_{\epsilon}(\xb,\yb)\,w(\xb,\yb) d\yb\,d\xb +
\sum_{\underset{J \ne I}{J=1}}^{N_p} \;\int\limits_{\Gamma_{JI}} \int\limits_{\;\,\left(\Omega_{I} \cup \Gamma_{II}\right) \cap \widetilde{\Gamma}_J}  \gamma_{\epsilon}(\xb,\yb)\,w(\xb,\yb) d\yb\,d\xb\right) \nonumber\\
& = 
\sum_{I=1}^{N_p}\left(\quad\iint\limits_{(\Omega_{I} \cup \Gamma_{II})^2} \gamma_{\epsilon}(\xb,\yb)\,w(\xb,\yb) d\yb\,d\xb +
\sum_{\underset{J \ne I}{J=1}}^{N_p} \;\int\limits_{\Gamma_{JI}} \int\limits_{\Gamma_{IJ}}  \gamma_{\epsilon}(\xb,\yb)\,w(\xb,\yb) d\yb\,d\xb \right). \label{parallel_integral}
\end{align}
A few remarks are in order. In the last equality the first double integral relies on information fully available on process $I$. In contrast, the outer integrals of the second term are defined on sub-domains $\Gamma_{JI}$, which belong to processes different from $I$. Thus, for process $I$, this information has to be made available through parallel implementation. Namely, before integration, each process $J$ sends to $I$ the needed information stored in the elements that make up $\Gamma_{JI}$ through the MPI send/receive protocol \cite{gropp1999using}. The inner integral is on the sub-region $\Gamma_{IJ}$ owned by process $I$. In the definition of $A_{ij}^{21}$ and $A_{ij}^{22}$, the inner integral identifies the $i$-th row of the matrix through the test function $\varphi_i(\yb)$. Hence, the degree of freedom of that row is owned by process $I$. We allocate the memory to store the entries of each row of the sparse parallel matrix on the process that owns the row itself; this results in an optimal parallel assembly since it minimizes (almost removes) the communication time. 

%%%
\begin{remark}
As already pointed out in Section \ref{AppDistSec}, identifying if a given element intersect a given region can be challenging and computationally expensive. As described before, we rely on a simple and fast conservative test to identify if an element $\mathcal{E}_j$, owned by process $J$ belongs to $\Gamma_{JI}$. 
Namely, we use formula \eqref{mindist} to approximate the minimum distance between the bounding box enveloping $\mathcal{E}_j$ and the bounding box enveloping all the elements owned by processor $I$. If the approximate distance is less than $\delta + \varepsilon$, then $\mathcal{E}_j$ is included into $\Gamma_{JI}$.
\end{remark}

%%%%%%%%%%%
\paragraph{Performance tests}
We first consider two-dimensional problems with both linear and quadratic FE
discretizations with $480641$ and $1440513$ degrees of freedom, respectively. We consider a quadrilateral mesh with $ml = 7$ (for which $ h = $1.5625E-03), $\varepsilon =h/2$, $\delta = 0.05$, $L_{max} = 3$, $L_{min} = 1$. Since the high computational costs are mainly due to the assembly procedure, we report separately the CPU times required by the assembly routine ($t_a$) and the total CPU time ($t_t$), which also includes the solver time. Computational times, as the number of processors $N_p$ grows, are reported in Table \ref{tab_scalability_linear}. 
We also report the time ratios $\hbox{TR}_a (N_p) = t_a(N_p / 2 ) / t_a(N_p ) $
and $\hbox{TR}_t (N_p) = t_r(N_p/ 2 ) / t_r(N_p )$. All cases feature large time rations, illustrating the good scaling properties of the parallel implementation for both linear and quadratic FE discretizations.

We also test the scaling properties of the parallel algorithm with three-dimensional numerical simulations. Similarly to the two-dimensional case, we consider a hexahedral mesh with $615489$ degrees of freedom and a quadratic FE discretization. We also set $ h = 0.025$, $\varepsilon = h/2$, 
$\delta = 0.1$, $L_{max} = 2$, $L_{min} = 1$. In Table \ref{tab_scalability_quadratic3D}
we report the assembly time $t_a$, the total time $t_t$ and the time ratios of the 
numerical simulation for increasing number of processes. As in the two-dimensional case, the parallel simulations show high time ratios,
indicating good scaling properties of the algorithm.

In Figure \ref{fig_parallel} we show the assembly speedup $S$ values referred to $N_p = 36$,
defined as $S = t_{a}(36) / t_a (N_p)$ for two-dimensional linear (left), quadratic (center)
and three-dimensional quadratic (right) simulations. We also report the linear speedup (dashed line).
On the basis of the presented results, the code shows excellent scalability properties especially in 3D. Thus, it proves to be suitable for large scale nonlocal simulations.

\begin{center}
\begin{table}[t]%
\centering
\caption{Assembly time $t_{a}$ and total time $t_t$ of simulations with $480641$ dofs, linear FEM,
and $1440513$ dofs, quadratic FEM,
 $ h = $1.5625E-03,
 $\varepsilon = h/2$, 
 $\delta = 0.05$,
 $L_{max} = 3$,
 $L_{min} = 1$.}
 \label{tab_scalability_linear}
 \begin{tabular*}{220pt}{@{\extracolsep\fill}c | c c | c c@{\extracolsep\fill}}
 \toprule
  & \multicolumn{2}{c|}{ Linear} &  \multicolumn{2}{c}{ Quadratic} \\ 
\midrule
$\mathbf{N_p}$    &    $\mathbf{t_a \, [s]}$     & $\mathbf{t_t \, [s]}$  &    $\mathbf{t_a \, [s]}$     & $\mathbf{t_t \, [s]}$ \\
\midrule
        36   &   4268.57 &   4470.87     &   1795.86 &   2168.10     \\
         &   \textbf{1.89} & \textbf{1.90} & \textbf{1.77} & \textbf{1.68} \\
        72   &   2251.59 &   2346.46     &   1013.76 &   1290.14     \\
         &   \textbf{1.92} & \textbf{1.94} & \textbf{1.88} & \textbf{1.94} \\
        144  &   1166.77 &   1207.84     &   538.12 &  663.53    \\
         &   \textbf{1.72} & \textbf{1.73} &   \textbf{1.64} & \textbf{1.70} \\
        288  &   675.67  &   694.44      &   327.75  &   388.88     \\
         &   \textbf{1.85} & \textbf{1.85} &   \textbf{1.65} & \textbf{1.72} \\
        576  &   364.65  &   373.77      &   197.83  &   225.32      \\
\bottomrule
\end{tabular*}
\end{table}
\end{center}

\begin{center}
\begin{table}[t]%
\centering
 \caption{Assembly time $t_{a}$, total time $t_t$ and time ratios of simulations with $615489$ dofs, quadratic FEM,
 $ h = 0.025$,
 $\varepsilon = h/2$, 
 $\delta = 0.1$,
 $L_{max} = 2$,
 $L_{min} = 1$.}
 \label{tab_scalability_quadratic3D}
 \begin{tabular*}{130pt}{@{\extracolsep\fill}c c c@{\extracolsep\fill}}
\toprule
$\mathbf{N_p}$    &    $\mathbf{t_a \, [s]}$     & $\mathbf{t_t \, [s]}$ \\
       \hline \\[-10pt]
%        1    &   8603.01  &  10584.7   \\
%        4    &   2190.39  &  2690.9    \\
%        8    &   1315.97  &  1656.2    \\
%        12   &   996.825  &  1261.6    \\
%        16   &   796.236  &  1033.3    \\
%        20   &   963.143  &  1196.4    \\
%        24   &   783.802  &  972.8     \\
%         9   &   5994.26  &  7235.5    \\
%         &   \textbf{1.82} & \textbf{1.81} \\
        % 18   &     &       \\
        %  &   \textbf{1.82} & \textbf{1.84} \\
        36   &   17598.5 &    17727.83     \\
         &   \textbf{1.92} & \textbf{1.92} \\
        72   &   9155.14 &   9223.08 \\
         &   \textbf{2.00} & \textbf{2.00} \\
        144  & 4556.1   &     4591.51 \\
         &   \textbf{1.99} & \textbf{1.97} \\
        288  &   2286.17  &   2314.36    \\
%        &   \textbf{1.65} & \textbf{1.72} \\
%        576  &   197.83  &   225.32      \\
\bottomrule
\end{tabular*}
\end{table}
\end{center}

\begin{figure}[t!]
   \centering
   \includegraphics[width=1.9in]{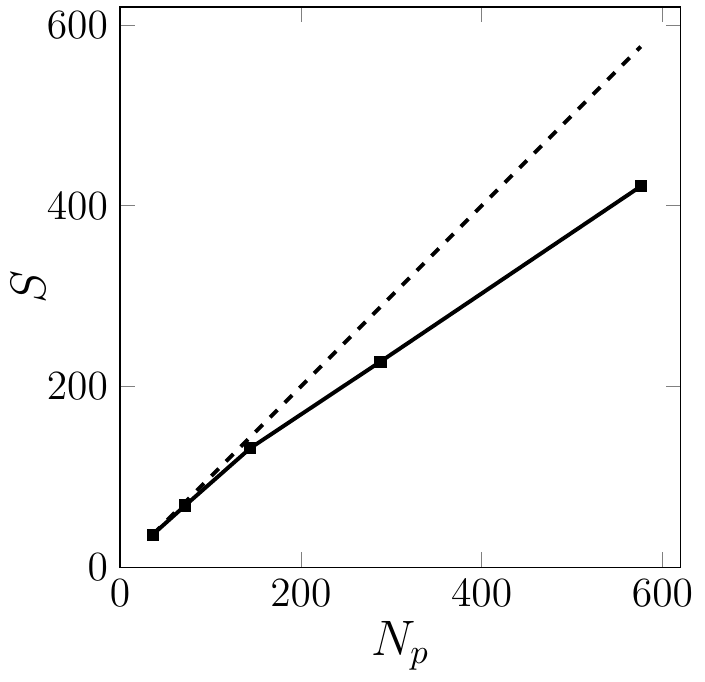}
   \quad
   \includegraphics[width=1.9in]{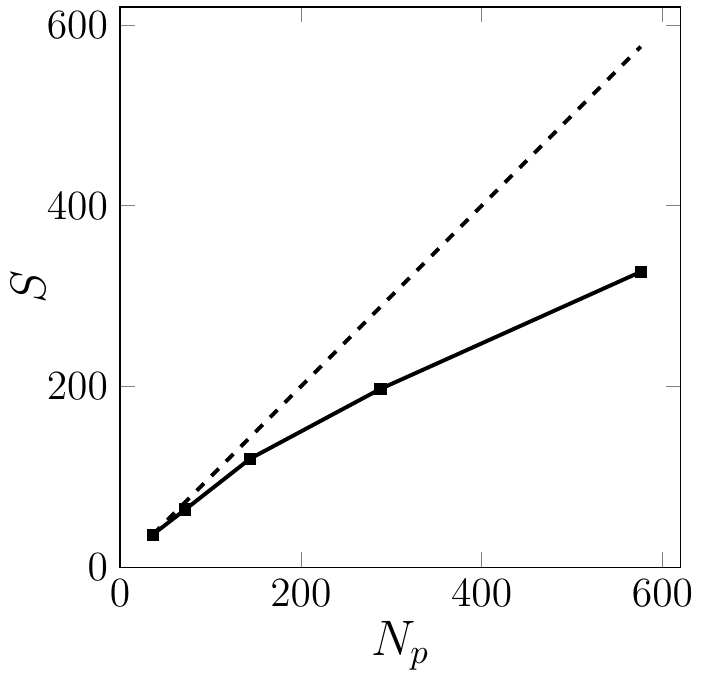}
   \quad
   \includegraphics[width=1.9in]{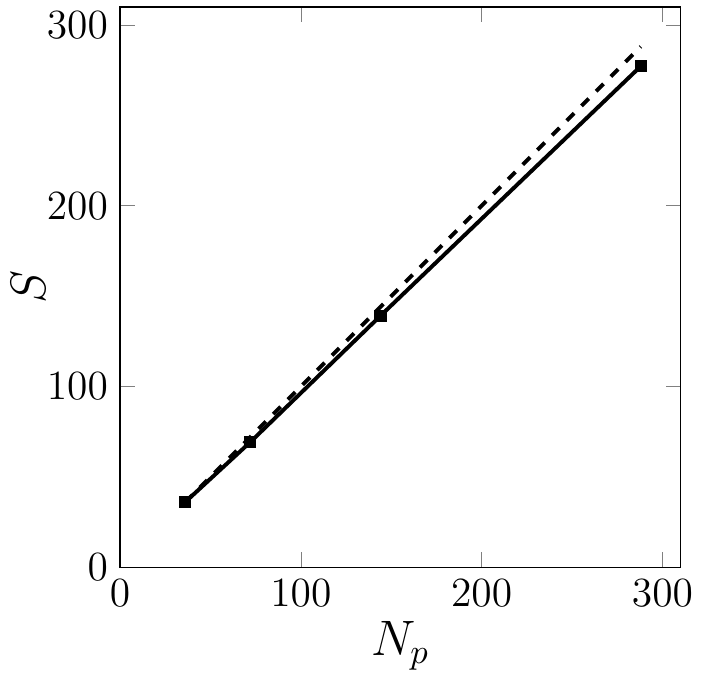}
\caption{Speedup as a function of the number of processors for two-dimensional
linear (left) and quadratic FE (center), and three-dimensional quadratic FE (right). %\gc{si possono fare le linee piu' thick e i numeri della legenda piu grossi?}
}
   \label{fig_parallel}
\end{figure}

%\begin{center}
%\begin{table}[t]%
%\centering
%\caption{This is sample table caption.\label{tab2}}%
%\begin{tabular*}{500pt}{@{\extracolsep\fill}lcccc@{\extracolsep\fill}}
%\toprule
%\textbf{col1 head} & \textbf{col2 head}  & \textbf{col3 head}  & \textbf{col4 head}  & \textbf{col5 head} \\
%\midrule
%col1 text & col2 text  & col3 text  & col4 text  & col5 text\tnote{$\dagger$}   \\
%col1 text & col2 text  & col3 text  & col4 text  & col5 text   \\
%col1 text & col2 text  & col3 text  & col4 text  & col5 text\tnote{$\ddagger$}   \\
%\bottomrule
%\end{tabular*}
%\begin{tablenotes}
%\item Source: Example for table source text.
%\item[$\dagger$] Example for a first table footnote.
%\item[$\ddagger$] Example for a second table footnote.
%\end{tablenotes}
%\end{table}
%\end{center}

%%%%%%%%%%%%%%%%%%%%%%%%%%%%%%%%%%%%%%%%%%%%%%%%%%%%%%%%%%%%%
%%%%%%%%%%%%%%%%%%%%%%%%%%%%%%%%%%%%%%%%%%%%%%%%%%%%%%%%%%%%%
\section{Conclusion}\label{sec:conclusion}
We introduced an efficient, flexible, and scalable algorithm for FE discretizations of nonlocal diffusion problems characterized by compactly supported kernels. The novelty of our work is three fold. First, we circumvent the numerical difficulties arising from the integration of discontinuous kernels by multiplying the integrand function by a mollifier. This allows us to avoid the tedious and costly process of identifying ball-element intersections, hence avoiding integration over partial elements and/or curved elements. These tasks, even though manageable in two-dimensional settings, become nontrivial in three dimensions. Second, we introduce an adaptive quadrature rule that allows us accurately resolve the high gradients featured by the integrand function without compromising the computational efficiency. Third, we propose a parallel implementation of the mollified, adaptive algorithm that shows excellent scalability properties especially in three dimensions, where the cost of numerical integration dominates.
Our numerical results illustrate the theoretical findings and show that, in two dimensions, our algorithm is competitive with other efficient approximations of FE implementations. In three dimensions, this is the first efficient and scalable parallel implementation that has no constraints on the type of mesh or FE spaces. 

As such, this work contributes to making variational discretizations of nonlocal models a viable option, even for large scale problems. It also represents an effort towards increasing the usability of nonlocal equations, for which the high computational costs often hinders their popularity in engineering contexts, despite their undeniable improved accuracy.

\section*{Supporting information}
This work was supported by the National Science Foundation (NSF) Division of Mathematical Sciences (DMS) program, project 1912902, by Sandia National Laboratories (SNL) Laboratory-directed Research and Development (LDRD) program, project 218318 and by the U.S. Department of Energy, Office of Advanced Scientific Computing Research under the Collaboratory on Mathematics and Physics-Informed Learning Machines for Multiscale and Multiphysics Problems (PhILMs) project. Sandia National Laboratories is a multimission laboratory managed and operated by National Technology and Engineering Solutions of Sandia, LLC., a wholly owned subsidiary of Honeywell International, Inc., for the U.S. Department of Energy's National Nuclear Security Administration contract number DE-NA0003525. This paper, SAND2021-0672, describes objective technical results and analysis. Any subjective views or opinions that might be expressed in the paper do not necessarily represent the views of the U.S. Department of Energy or the United States Government.

%\nocite{*}% Show all bib entries - both cited and uncited; comment this line to view only cited bib entries;
\bibliographystyle{plain}
\bibliography{bib.bib}%

\end{document}